\documentclass[10pt,a4paper]{article}
\usepackage{amsmath,amssymb,amsthm,mathrsfs,graphicx}
\usepackage[titletoc, title]{appendix}
\usepackage{mathrsfs}
\usepackage{paralist}
\usepackage{graphics} 
\usepackage{epsfig} 
\usepackage{multirow}
\usepackage{graphicx}
\usepackage{epstopdf}
\usepackage{syntonly,indentfirst,mathrsfs,latexsym,float}
\usepackage{bm,amsmath,amsbsy,amsthm,amssymb,amsfonts}
\usepackage[colorlinks,linkcolor=red]{hyperref}
\usepackage[capitalise, nosort]{cleveref}
\usepackage{array}
\usepackage{color}
\usepackage{multirow}
\usepackage{mathrsfs}
\usepackage{type1cm}
\usepackage{wasysym}

\numberwithin{equation}{section}

\newtheorem{myTheo}{Theorem}[section]
\newtheorem{myLem}{Lemma}[section]
\usepackage{indentfirst} 
\usepackage{stmaryrd}
\usepackage{subfigure}
\usepackage[utf8]{inputenc}
\usepackage{pgf,tikz}
\usepackage{threeparttable}
\usetikzlibrary{arrows}
\newcommand{\tabcaption}{\def\@captype{table}\caption}

\newtheorem{rem}{Remark}[section]

\newtheorem{exmp}{Example}[section]
\date{}
\begin{document}

\title{{\Large \bf Semi-discrete and fully discrete HDG methods for Burgers' equation}
\thanks
  {
    This work was supported in part by National Natural Science Foundation
    of China (11771312).
  }
}     
\author{
  Zimo Zhu \thanks{Email:  zzm@stu.scu.edu.cn}, \ 
  Gang Chen \thanks{Email:  cglwdm@scu.edu.cn}
  Xiaoping Xie \thanks{Corresponding author. Email: xpxie@scu.edu.cn} \\
  {School of Mathematics, Sichuan University, Chengdu 610064, China}
}

\date{}
\maketitle

	\maketitle
\begin{abstract}
	This paper proposes  semi-discrete and fully discrete hybridizable discontinuous Galerkin (HDG)  methods for the Burgers' equation in two and three dimensions. In the spatial discretization, we use piecewise polynomials of degrees $ k \ (k \geq 1), k-1$ and $ l \ (l=k-1; k)  $ to approximate  the scalar function, flux variable and the interface trace of  scalar function,  respectively.  In the   full discretization method,  we apply a backward Euler scheme for the temporal discretization. Optimal a priori error estimates are derived. Numerical
experiments are presented to support the theoretical results.
	
	 \end{abstract}


\section{Introduction}
Let $ \Omega\subset \mathbb{R}^{d}(d=2,3) $ be a   polyhedral  domain with    boundary $ \partial\Omega $, and  let $ T > 0 $ be a given final time. We consider the following   Burgers’ equation  \cite{Burgers1948,Burgers1995}:
\begin{subequations}
\begin{align}
u_{t}-\nu\Delta u+\bm{b}(u)\cdotp\nabla u&=f\ \ \text{in}\ \  \Omega\times[0,T),\label{1.1a}\\  
u&=0\ \ \text{on} \ \ \partial \Omega\times[0,T),\label{1.1b}\\
u(\cdotp,0)&=u_{0}\ \ \text{in} \ \ \Omega,  \label{1.1c}  
\end{align}
\end{subequations}
where  $ u(x,t) $ is the unknown scalar function with   initial value $ u_{0}(x) $, $ \bm{b}(u)=[u,u]^{T} $ when $ d=2 $ and $ \bm{b}(u)=[u,u,u]^{T} $ when $ d=3 $, $ \nu >0$ is the coefficient of viscosity, and 
$ f (x,t)  $ is the prescribed force.  

The Burgers’ equation, which  can be viewed as a simplified model of Navier-Stokes equation, is a nonlinear partial differential equation that simulates the propagation and reflection of shock waves. It is widely used in many physical fileds such as fluid mechanics, nonlinear acoustics, and gas dynamics. %Particularly, Burgers’ equation. %Therefore the research of Burgers' equation is also helpful to Navier–Stokes problems. 
In recent  decades, there have developed   many finite element (FE) methods  for Burgers’ equation, such as
%numerical methods for Burgers’ equation, such as the finite difference method \cite{Bahadr2003,Fletcher1983,Radwan2005}, the finite volume method \cite{Yang2012}, 
  conforming  
  methods \cite{Aksan2005,Arminjon1978,%Arminjon1979,
  Burman2015,Caldwell1981,%Caldwell1987,
  Dogan2004,Lucchi2010,%Luo2009,
  Ozics2003%,Varoglu2010
  }, B-spline   methods \cite{Ali1992,  %Arora2013,
Mittal2015,Ucar2019}, least-squares  methods \cite{Kutluay2004,Sterck2005,Winterscheidt2010}, mixed  methods \cite{Chen2004,Feng2014,Feng2016,Pany2007,ShiDong2013}, 
discontinuous Galerkin (DG)  methods \cite{Arminjon1981,Shao2011,Zhao2011},  and weak Galerkin methods \cite{Chen2019,Hussein2020}.

%In recent years, the discontinuous Galerkin (DG) finite element methods have become popular due to their computational
%flexibility and their ability to incorporate physical properties. In \cite{Arminjon1981}, two discontinuous finite element methods with piecewise linear functions were used to solve the inviscid Burgers’ equation and the full equation with viscosity, respectively. The local discontinuous Galerkin (LDG) methods were applied to solve the Burger's equation in \cite{Shao2011,Zhao2011} based on the
%Hopf–Cole transformation. 

The hybridizable discontinuous
Galerkin (HDG) framework, presented in \cite{Cockburn2009} for second order elliptic problems, provides
a unifying strategy for hybridization of finite element methods.   
By the local elimination of the unknowns defined in the interior of elements, the HDG method    leads to a system where the unknowns are only the globally coupled degrees of freedom describing the introduced Lagrange multiplier. 
 In \cite{Cockburn20092}, an implicit high-order HDG method was presented
for nonlinear convection–diffusion equations. The HDG method used piecewise polynomials of degrees $ k(k\geq 0) $ for the approximations of the scalar variable, corresponding flux and trace of the scalar variable. A numerical example for the  two dimensional Burgers' equation with the third-order backward difference formula in time discretization was presented.   We refer the readers to 
  \cite{%chen2015robust,
  chen2014robust,Cockburn2011HDGstokes,  HanChenWangXie2019Extended, Li-X2016analysis,%Li-X2016SIAM,
  %Li-X-Z2016analysis,
  Qiu2016An}
for some     developments  and applications of the HDG method. 

%The approximations for the scalar variable and the flux that converged to the optimal order of $ k + 1 $ in the $ L^{2} $-norm were observed. 
%
%Similar to the HDG method, the weak Galerkin (WG) finite element method first proposed in \cite{Wang2013} also has
%% these advantages. 
% We
%note that in some special cases the WG method and the HDG method are equivalent \cite{XieC2016,XieC2017,Xie2016elastic}. In \cite{Chen2019,Hussein2020}, the WG methods were presented to solve one dimensional Burgers’ equation and one dimensional coupled Burgers’ equations, respectively. Both the optimal order error estimates in the discrete $ H^{1} $-norm and $ L^{2} $-norm were derived.
%
%
%We also refer to \cite{Cockburn2019,Cockburn2011,Qiu2016} for some HDG methods and \cite{Han20192,Han2019,Ye2019,Wang2019} for some WG methods for the nonlinear equations.

In this paper, we consider an  HDG discretization of the Burgers' equation in two and three dimensions. In the spatial
discretization,  the flux variable $ \bm{q} $, the scalar variable $ u $ and its trace are approximated respectively by piecewise polynomials of degrees $ k-1, k  $ and $ l (l=k,k-1)$ with  $k\geq 1$. In the fully discrete scheme,  a backward Euler scheme is adopted for the temporal derivative. 

The   rest of the paper is arranged as follows. In section 2, we introduce some notations and the  weak problem. Section 3 presents the semi-discrete HDG scheme, prove the existence and uniqueness of the solution and carries out  the error analysis. Section 4 discusses the fully discrete HDG scheme, including the stability, the existence and uniqueness of the solution and the  error estimation. Finally, we provides some numerical examples to verify the theoretical results in section 5.

\section{Notation and weak problem}
%\subsection{Notation}
For any bounded domain $ D\subset \mathbb{R}^{s}(s=d,d-1) $ and integer $m\geq 0$, let $ H^{m}(D) $ and $ H^{m}_{0}(D) $ denote the usual $ m^{th} $-order Sobolev spaces on $ D $, and $ \|\cdotp\|_{m,D} $, $ |\cdotp|_{m,D} $ denote the norm and semi-norm on these spaces, respectively. We use $ (\cdotp,\cdotp)_{m,D} $ to denote the inner product of $ H^{m}(D) $, with $ (\cdotp,\cdotp)_{D}:=(\cdotp,\cdotp)_{0,D} $. When $ D=\Omega $, we set $ \|\cdotp\|_{m}:=\|\cdotp\|_{m,\Omega} $, $ |\cdotp|_{m}:=|\cdotp|_{m,\Omega} $, and $ (\cdotp,\cdotp):=(\cdotp,\cdotp)_{\Omega} $. In particular, when $ D\subset \mathbb{R}^{d-1} $, we use $ \left\langle \cdotp,\cdotp\right\rangle _{D} $ to replace $ (\cdotp,\cdotp)_{D} $.  We denote by $P_m(D)$    the set of all polynomials on $D$ with degree at most $m$.

%Assume  that $\Omega$ is a polygonal/polyhedral domain. 
Let $\mathcal{T}_h=\cup\{K\}$, consisting of arbitrary open polygons/polyhedrons, be a  partition of the domain $\Omega$. For any $K\in \mathcal{T}_h$,   let $h_K$ be the infimum of the diameters of circles (or spheres) containing $K$ and denote
by $h := \max_{K\in \mathcal{T}_h} h_K$ the mesh size.   We assume that $\mathcal{T}_h$ is shape-regular% and $\partial \mathcal{T}_h$ be the union of $\partial K$ for $K\in \mathcal{T}_h$.
in the   sense that the following two assumptions hold (cf. \cite{Xie2016elastic}): %,  shape regular under  hold true.
\begin{enumerate}
\item[(M1)] There exists a positive constant $\theta_*$ such that the following holds: for each element $K\in \mathcal{T}_h$, there exists a point $M_K\in K$ such that $K$ is star-shaped with respect to every point in the circle (or sphere) of center $M_K$ and radius $\theta_*h_K$.
	
\item[(M2)] There exists a positive constant $l_*$ such that for every element $K\in \mathcal{T}_h$, the distance between any two vertexes is no less than $  l_*h_K$. 
\end{enumerate}

We denote by $ \varepsilon_{h} $ the set of all faces in the mesh, and by  $ \varepsilon_{h}^{o} $ and $ \varepsilon_{h}^{\partial} $ the sets of interior edges/faces and boundary edges/faces, respectively. 
%For each face $ e $ we say $ e\in \varepsilon_{h}^{o} $ is an 
%interior face if the Lebesgue measure of 
%$ e = \partial K^{+}\cap \partial K^{-} $ for some pair of elements $ K^{+}, K^{-} \in \mathcal{T}_{h} $ is
%non-zero, similarly, $ e \in \varepsilon_{h}^{\partial} $ is a boundary face if the Lebesgue measure of $ e = \partial K\cap \partial\Omega $ is non-zero. 
For any %$ K\in\mathcal{T}_{h} $, 
$ e\in\varepsilon_{h} $, we denote by %$ h_{K} $ and
 $ h_{e} $ the diameter of   $ e $. We also introduce the following mesh-dependent inner products and norms:
\begin{equation*}\begin{aligned}
(u,v)_{\mathcal{T}_{h}}&:=\sum\limits_{K\in\mathcal{T}_{h}}(u,v)_{K},\ \ \ \ \ \  \|u\|^{2}_{0,\mathcal{T}_{h}}:=\sum\limits_{K\in\mathcal{T}_{h}}\|u\|^{2}_{0,K},\\
\left\langle u,v\right\rangle_{\partial\mathcal{T}_{h}} &:=\sum\limits_{K\in\mathcal{T}_{h}}\left\langle u,v\right\rangle_{\partial K},\ \ \|u\|^{2}_{0,\partial\mathcal{T}_{h}}:=\sum\limits_{K\in\mathcal{T}_{h}}\|u\|^{2}_{0,\partial K}.
\end{aligned}
\end{equation*}

For convenience, we use the notation $a\lesssim b$ to denote that there exists a generic positive constant $C$, independent of the spatial and temporal mesh parameters, $h$ and $ \triangle t$, %and the Lam\'e parameters, 
such that $a\leq Cb.$

%\subsection{Basic results for Burgers' equation}
%We denote by $ A $ the linear unbounded operator in $ L^{2}(\Omega) $ which is associated with $ H_{0}^{1}(\Omega) $ and $ L^{2}(\Omega) $ such that
%We denote by $ A $ the Laplace operator such that
%\begin{equation*}
%-(Au,v)=(\nabla u,\nabla v),\ \ \ \forall u,v\in H_{0}^{1}(\Omega).
%\end{equation*}
%The domain of $ A $ in $ L^{2}(\Omega) $ is denoted by $ D(A) $; $ A $ is a self-adjoint positive operator in $ L^{2}(\Omega) $. Also, $ A $ is an isomorphism from $ D(A) $ onto $ L^{2}(\Omega) $. The space $ D(A) $ can be fully characterized by using the regulartiy theory of linear elliptic systems, see \cite{Temam1995}, 
%\begin{equation*}
%D(A)=H^{2}(\Omega)\cap H_{0}^{1}(\Omega),
%\end{equation*}
%furthermore, $ \|Au\|_{0} $ is on $ D(A) $ a norm equivalent to that induced by $ H^{2}(\Omega) $.

To give the weak problem for the model (\ref{1.1a})-(\ref{1.1c}), we need to introduce the    bilinear form $ \mathcal{A}(\cdotp,\cdotp) $ and    trilinear form $\mathcal{B}(\cdot,\cdot,\cdot)$:  %on $ H_{0}^{1}(\Omega)\times H_{0}^{1}(\Omega) $ 
\begin{equation*}
\mathcal{A}(u,v):=\nu(\nabla u,\nabla v),\ \ \ \forall u,v\in H_{0}^{1}(\Omega),
\end{equation*}
%and  define a trilinear form $\mathcal{B}(\cdot,\cdot,\cdot)$ on $ H_{0}^{1}(\Omega)\times H_{0}^{1}(\Omega)\times H_{0}^{1}(\Omega) $ by
\begin{equation*}
\mathcal{B}(u,v,w):=-\frac{1}{3}(\bm{b}(u),v\nabla w)+\frac{1}{3}(\bm{b}(u),w\nabla v),\ \ \ \forall u,v,w\in H_{0}^{1}(\Omega).
\end{equation*}
Since  
\begin{equation*}
\begin{aligned}
(\bm{b}(u)\cdotp\nabla u,v)&=\frac{1}{3}(\nabla\cdotp\bm{b}(u),uv)+\frac{2}{3}(\bm{b}(u),v\nabla u)\\
&=-\frac{1}{3}(\bm{b}(u),\nabla(uv))+\frac{2}{3}(\bm{b}(u),v\nabla u)\\
&=-\frac{1}{3}(\bm{b}(u),u\nabla v)+\frac{1}{3}(\bm{b}(u),v\nabla u) \quad  \forall  v\in H_{0}^{1}(\Omega) ,
\end{aligned}
\end{equation*}
it is easy to see that  %$ \mathcal{B} $ satisfies the following important properties in two dimensions(see \cite{Temam1984}, Lemma 3.4):
\begin{align}
\mathcal{B}(u,v,w)&=-\mathcal{B}(u,w,v) \quad \forall u,v,w\in H_{0}^{1}(\Omega).
%|\mathcal{B}(u,v,w)|&\leq c_{1}\|u\|^{\frac{1}{2}}_{0}\|u\|^{\frac{1}{2}}_{1}\|v\|_{1}\|w\|^{\frac{1}{2}}_{0}\|w\|^{\frac{1}{2}}_{1},
\end{align}
%for all $ u,v,w\in H_{0}^{1}(\Omega) $ and
%\begin{equation}
%|\mathcal{B}(u,v,w)|\leq c_{2}\|u\|^{\frac{1}{2}}_{0}\|u\|^{\frac{1}{2}}_{1}\|v\|^{\frac{1}{2}}_{1}\|Av\|^{\frac{1}{2}}_{0}\|w\|^{\frac{1}{2}}_{0}\|w\|^{\frac{1}{2}}_{1},
%\end{equation}
%for all $ u,w\in H_{0}^{1}(\Omega), v\in D(A) $, where $ c_{1} $ and $ c_{2} $ are positive constants depending on $ \Omega $.

With the above notations, the weak form of (\ref{1.1a})-(\ref{1.1c}) is given as follows: find $ u(t)\in H_{0}^{1}(\Omega) $ such that 
\begin{equation}\label{2.4}
\left\{
\begin{aligned}
(u_{t},v)+\mathcal{A}(u,v)+\mathcal{B}(u,u,v)&=(f,v),\ \ \forall v\in H_{0}^{1}(\Omega),\  t\in (0,T]\\
u(0)&=u_{0},\ \ \ x\in \Omega.  
\end{aligned}
\right.
\end{equation} 

From \cite[Section III.Theorem 3.1]{Temam1997}, it holds the following wellposedness result for  the weak problem \eqref{2.4}.
\begin{myLem}
In the case of $d=2$,  given $ f \in L^{\infty}(0,T;L^{2}(\Omega)) $ and $ u_{0}(x) \in L^{2}(\Omega) $,  the problem (\ref{2.4})  admits   a unique solution $ u $  satisfying
\begin{equation*}
u\in C([0,T];L^{2}(\Omega))\cap L^{2}(0,T;H_{0}^{1}(\Omega)).
\end{equation*}
Furthermore, if $ u_{0}(x)\in H_{0}^{1}(\Omega) $, then
\begin{equation*}
u\in C([0,T];H_{0}^{1}(\Omega))\cap L^{2}(0,T;H^{2}(\Omega)\cap H_{0}^{1}(\Omega)) .
\end{equation*}
\end{myLem}

\section{Semi-discrete HDG  method}
\subsection{Semi-discrete scheme}

Introduce a new variable $ \bm{q}=-\nabla u $,then  we  rewrite (\ref{1.1a})-(\ref{1.1c}) as
\begin{equation} \label{2.5}
	\left\{
	\begin{aligned}
\bm{q}+\nabla u &= 0,\ \ \text{in}\ \ \Omega\times (0,T],\\
u_{t}+\nu\nabla\cdotp\bm{q}+\bm{b}(u)\cdotp\nabla u&=f,\ \ \text{in}\ \ \Omega\times (0,T],\\
u&=0,\ \ \text{on} \ \ \partial \Omega\times[0,T],\\
u(\cdotp,0)&=u_{0},\ \ \text{in} \ \ \Omega.
\end{aligned} 
\right.
\end{equation}

For any integer $ k\geq 1  $ and  $ l=k-1,k $, we introduce the following  finite element spaces: % $ \bm{Q}_{h}, V_{h}, \widehat{V}_{h} $, 
\begin{equation*}
\begin{aligned}
\bm{Q}_{h}&:=\{\bm{v}_{h}\in [L^{2}(\Omega)]^{d}:\bm{v}_{h}|_{K}\in [P_{k-1}(K)]^{d},\forall K\in \mathcal{T}_{h}\},\\
V_{h}&:=\{v_{h}\in L^{2}(\Omega):v_{h}|_{K}\in P_{k}(K),\forall K\in \mathcal{T}_{h}\},\\
\widehat{V}_{h}&:=\{\widehat{v}_{h}\in L^{2}(\varepsilon_{h}):\widehat{v}_{h}|_{E}\in P_{l}(E),\forall E\in \varepsilon_{h},\widehat{v}_{h}|_{\varepsilon_{h}^{\partial}}=0\}.
\end{aligned}
\end{equation*}
%where $ P_{k}(K) $ denotes the set of polynomials of degree at most $ k $ on the element $ K $ (similarly
%$ P_{k}(\varepsilon_{h}) $ denotes the set of polynomials of degree at most $ k $ on the faces in the mesh).

Let us introduce  the standard $ L^{2} $ projections $ \boldsymbol{\Pi}_{k-1}^{o}:[L^{2}(\Omega)]^{d}\rightarrow \bm{Q}_{h} $, $ \Pi_{k}^{o}: L^{2}(\Omega)\rightarrow V_{h} $, and
$ \Pi_{l}^{\partial}: L^{2}(\varepsilon_{h})\rightarrow \widehat{V}_{h} $, which satisfy
\begin{equation*}
\begin{aligned}
(\boldsymbol{\Pi}_{k-1}^{o}\bm{q},\bm{r})_{K}&=(\bm{q},\bm{r})_{K},\ \ \ \ \forall \bm{r}\in [P_{k-1}(K)]^{d},\\
(\Pi_{k}^{o}u,w)_{K}&=(u,w)_{K},\ \ \ \forall w\in P_{k}(K),\\
\left\langle \Pi_{l}^{\partial}u,\mu\right\rangle _{e}&=\left\langle u,\mu \right\rangle_{e},\ \ \ \ \forall \mu\in P_{l}(e).
\end{aligned}
\end{equation*}
Then the  semi-discrete HDG scheme for (\ref{2.5}) reads as follows: find $ (\bm{q}_{h},u_{h},\widehat{u}_{h})\in \bm{Q}_{h}\times V_{h}\times\widehat{V}_{h} $ such that 
\begin{equation} \label{2.10}
\left\{
\begin{aligned}
\mathcal{A}_{h}(\bm{q}_{h},\bm{r})-\mathcal{C}_{h}(u_{h},\widehat{u}_{h};\bm{r})&=0,\\
(u_{h,t},w)_{\mathcal{T}_{h}}+\nu \mathcal{C}_{h}(w,\mu;\bm{q}_{h})
+\mathcal{S}_{h}(u_{h},\widehat{u}_{h};w,\mu)
+\mathcal{B}_{h}(u_{h};u_{h},\widehat{u}_{h};w,\mu)&=(f,w)_{\mathcal{T}_{h}},\\
u_{h}(0)&=\Pi_{k}^{o}u_{0}, 
\end{aligned}
\right.
\end{equation}
 for all $ (\bm{r},w,\mu)\in \bm{Q}_{h}\times V_{h}\times\widehat{V}_{h} $. Here 
%\begin{equation}
%\nu(\bm{q}_{h},\bm{r})_{\mathcal{T}_{h}}-\nu(u_{h},\nabla\cdotp \bm{r})_{\mathcal{T}_{h}}+\nu\left\langle \widehat{u}_{h},\bm{r}\cdotp\bm{n}\right\rangle_{\partial\mathcal{T}_{h}}=0, \label{2.7}
%\end{equation}
%\begin{equation}
%\begin{aligned}
%(u_{h,t},w)_{\mathcal{T}_{h}}&+\nu(\nabla\cdotp \bm{q}_{h},w)_{\mathcal{T}_{h}}-\nu\left\langle \bm{q}_{h}\cdotp\bm{n},\mu\right\rangle_{\partial\mathcal{T}_{h}}\\
%&+\nu\left\langle\tau(\Pi_{l}^{\partial}u_{h}-\widehat{u}_{h}),\Pi_{l}^{\partial}w-\mu \right\rangle_{\partial\mathcal{T}_{h}}\\
%&-\frac{1}{3}(\bm{b}(u_{h})u_{h},\nabla w)_{\mathcal{T}_{h}}+\frac{1}{3}(\bm{b}(u_{h})\cdotp\nabla u_{h},w)_{\mathcal{T}_{h}}\\
%&-\frac{1}{3}\left\langle \bm{b}(u_{h})\cdotp\bm{n}u_{h},\mu\right\rangle_{\partial\mathcal{T}_{h}}+\frac{1}{3}\left\langle \bm{b}(u_{h})\cdotp\bm{n}\widehat{u}_{h},w\right\rangle_{\partial\mathcal{T}_{h}}=(f,w)_{\mathcal{T}_{h}},  \label{2.8}
%\end{aligned}
%\end{equation}
\begin{equation*}
\begin{aligned}
\mathcal{A}_{h}(\bm{q}_{h},\bm{r})&:=(\bm{q}_{h},\bm{r})_{\mathcal{T}_{h}},\\
\mathcal{C}_{h}(u_{h},\widehat{u}_{h};\bm{r})&:=(u_{h},\nabla\cdotp \bm{r})_{\mathcal{T}_{h}}-\left\langle \widehat{u}_{h},\bm{r}\cdotp\bm{n}\right\rangle_{\partial\mathcal{T}_{h}},\\
\mathcal{S}_{h}(u_{h},\widehat{u}_{h};w,\mu)&:=\nu\left\langle\tau(\Pi_{l}^{\partial}u_{h}-\widehat{u}_{h}),\Pi_{l}^{\partial}w-\mu \right\rangle_{\partial\mathcal{T}_{h}},\\
\mathcal{B}_{h}(v_{h};u_{h},\widehat{u}_{h};w,\mu)&:=-\frac{1}{3}(\bm{b}(v_{h})u_{h},\nabla w)_{\mathcal{T}_{h}}+\frac{1}{3}(\bm{b}(v_{h})\cdotp\nabla u_{h},w)_{\mathcal{T}_{h}}\\
&\ \ \ -\frac{1}{3}\left\langle \bm{b}(v_{h})\cdotp\bm{n}u_{h},\mu\right\rangle_{\partial\mathcal{T}_{h}}+\frac{1}{3}\left\langle \bm{b}(v_{h})\cdotp\bm{n}\widehat{u}_{h},w\right\rangle_{\partial\mathcal{T}_{h}},
\end{aligned}
\end{equation*}
and   $ \tau|_{\partial K}=h_{K}^{-1} $.

%Then (\ref{2.7})-(\ref{2.8}) can be rewritten as: find $ (\bm{q}_{h},u_{h},\widehat{u}_{h})\in \bm{Q}_{h}\times V_{h}\times\widehat{V}_{h} $ such that, for all $ (\bm{r},w,\mu)\in \bm{Q}_{h}\times V_{h}\times\widehat{V}_{h} $ satisfying

We introduce an operator $ \mathcal{K}_{h}:V_{h}\times\widehat{V}_{h}\rightarrow \bm{Q}_{h} $ defined by
\begin{equation*}
(\mathcal{K}_{h}(w,\mu),\bm{r})_{\mathcal{T}_{h}}:=-(w,\nabla_{h}\cdotp \bm{r})_{\mathcal{T}_{h}}+\left\langle \mu,\bm{r}\cdotp\bm{n}\right\rangle_{\partial\mathcal{T}_{h}},\ \forall \bm{r}\in \bm{Q}_{h}.
\end{equation*}
It is easy to see that $ \mathcal{K}_{h} $ is well defined. From (\ref{2.10}) we can immediately get
\begin{equation}\label{qh-Kh}
\bm{q}_{h}=-\mathcal{K}_{h}(u_{h},\widehat{u}_{h}).
\end{equation}
Hence,     the system (\ref{2.10}) can be reduced to   the following one:
find $ (\bm{q}_{h},u_{h},\widehat{u}_{h})\in \bm{Q}_{h}\times V_{h}\times\widehat{V}_{h}  $ such that, for all $ (w,\mu)\in   V_{h}\times\widehat{V}_{h} $, 
\begin{equation}
\left\{
\begin{aligned}
\bm{q}_{h}+\mathcal{K}_{h}(u_{h},\widehat{u}_{h})&=0,\\
(u_{h,t},w)_{\mathcal{T}_{h}}+\nu(\mathcal{K}_{h}(u_{h},\widehat{u}_{h}),\mathcal{K}_{h}(w,\mu))_{\mathcal{T}_{h}}+\mathcal{S}_{h}(u_{h},\widehat{u}_{h};w,\mu)
+\mathcal{B}_{h}(u_{h};u_{h},\widehat{u}_{h};w,\mu)&=(f,w)_{\mathcal{T}_{h}},\\
u_{h}(0)&=\Pi_{k}^{o}u_{0}. \label{2.12}
\end{aligned}
\right.
\end{equation}

\begin{rem}
	We note that the degree of approximation polynomials in $ \bm{Q}_{h} $ can be chosen as $ k $. Such a
	choice makes no difference in the subsequent analysis, and then leads to the same convergence rates as
	the degree $ k-1 $.  
\end{rem}
%\begin{myLem}
%	For any $ (u_{h},\widehat{u}_{h},w,\mu)\in \bm{Q}_{h}\times V_{h}\times\widehat{V}_{h}\times \bm{Q}_{h}\times V_{h}\times\widehat{V}_{h} $, we have
%	\begin{equation*}
%	b_{h}(u_{h};u_{h},\widehat{u}_{h};u_{h},\widehat{u}_{h})=0.
%	\end{equation*}
%\end{myLem}
%The proof of Lemma 2.1 is straightforward, hence we omit it here.

%\section{Well-posedness of the semi-discrete HDG scheme}

\subsection{Basic results}
We introduce the following semi-norm on $ V_{h}\times\widehat{V}_{h} $:  For $ \forall (u_{h},\widehat{u}_{h})\in V_{h}\times\widehat{V}_{h},$
\begin{equation*}
|\|(u_{h},\widehat{u}_{h})\||^{2}:=\|\mathcal{\mathcal{K}}_{h}(u_{h},\widehat{u}_{h})\|^{2}_{0,\mathcal{T}_{h}} +  \|\tau^{\frac{1}{2}}(\Pi_{l}^{\partial}u_{h}-\widehat{u}_{h})\|^{2}_{0,\partial\mathcal{T}_{h}}.
\end{equation*}
We can show that $ |\|(\cdotp,\cdotp)\|| $ is a norm.   In fact, if $ |\|(u_{h},\widehat{u}_{h})\||=0 $, then
$ \nabla u_{h}=0 $ and   $ \Pi_{l}^{\partial}u_{h}|_{\partial K}=u_{h}|_{\partial K}=\widehat{u}_{h}|_{\partial K}, $
which means that $ u_{h} $ is piecewise constant with respect to $ \mathcal{T}_{h}  $ and $ u_{h}=\widehat{u}_{h} $ on $ \varepsilon_{h} $. Since $ \widehat{u}_{h}=0 $ on $ \partial\Omega $, then $ u_{h}=\widehat{u}_{h}=0 $.

By using the trace theorem, the inverse inequality, and scaling arguments, we can easily get the following lemma.
\begin{myLem} \label{L3.1}
	\cite{Shi2013}For all $ K\in\mathcal{T}_{h}, \omega\in W^{1,\tilde{q}}(K) $, and $ 1\leq \tilde{q}\leq \infty $, we have 
	\begin{equation*}
	\|\omega\|_{0,\tilde{q},\partial K}\lesssim h_{K}^{-\frac{1}{\tilde{q}}}\|\omega\|_{0,\tilde{q},K}+h_{K}^{1-\frac{1}{\tilde{q}}}|\omega|_{1,\tilde{q},K}.
	\end{equation*}
	In particular, for all $ \omega\in P_{k}(K) $,
	\begin{equation*}
	\|\omega\|_{0,\tilde{q},\partial K}\lesssim h_{K}^{-\frac{1}{\tilde{q}}}\|\omega\|_{0,\tilde{q},K}.
	\end{equation*}
\end{myLem}

\begin{myLem} \label{L3.2}
	For any $ (w,\mu)\in V_{h}\times\widehat{V}_{h} $ we have
	\begin{equation}
	\|\nabla_{h} w\|_{0,\mathcal{T}_{h}}+\|\tau^{\frac{1}{2}}(w-\mu)\|_{0,\partial\mathcal{T}_{h}}\lesssim |\|(w,\mu)\||. \label{3.1}
	\end{equation}
\end{myLem}
\begin{proof}
	For any $ (w,\mu)\in V_{h}\times\widehat{V}_{h} $, by the definition of $ \mathcal{K}_{h} $, it holds
	\begin{equation*}
	\begin{aligned}
	(\mathcal{K}_{h}(w,\mu),\bm{r})_{\mathcal{T}_{h}}&=-(w,\nabla_{h}\cdotp \bm{r})_{\mathcal{T}_{h}}+\left\langle \mu,\bm{r}\cdotp\bm{n}\right\rangle_{\partial\mathcal{T}_{h}}\\
	&=(\nabla_{h}w,\bm{r})_{\mathcal{T}_{h}}+\left\langle \mu-w,\bm{r}\cdotp\bm{n}\right\rangle_{\partial\mathcal{T}_{h}},
	\end{aligned}
	\end{equation*}
	for all $ \bm{r}\in \bm{Q}_{h} $. By taking $ \bm{r}=\nabla_{h}w $, the estimate (\ref{3.1}) follows from the H\"{o}lder’s inequality and the inverse inequality.
\end{proof}
%The following embedding relationships are standard .

\begin{myLem}\cite{Shi2013} \label{L3.3}
	The embedding relationship 
	\begin{equation*}
	W^{1,2}(\Omega)\hookrightarrow W^{0,\tilde{q}}(\Omega)
	\end{equation*}
	holds for $ \tilde{q} $ satisfying $ 2\leq \tilde{q}<\infty $ when $ d=2 $, $ 2\leq \tilde{q}\leq6 $ when $ d=3 $.
\end{myLem}
\begin{myLem} \label{L3.3.1}
	\cite{Karakashian2007}There exists an interpolation operator, called Oswald interpolation, $ I_{h}: V_{h}\rightarrow V_{h}\cap H_{0}^{1}(\Omega) $, such that, for any $ w\in V_{h} $,
\begin{align} 
\sum_{K\in\mathcal{T}_{h}}\|w-I_{h}w\|^{2}_{0,K}&\lesssim \sum_{e\in\varepsilon_{h}}h_{e}\|[[w]]\|^{2}_{0,e}, \label{3.1.1}\\
\sum_{K\in\mathcal{T}_{h}}|w-I_{h}w|^{2}_{1,K}&\lesssim \sum_{e\in\varepsilon_{h}}h_{e}^{-1}\|[[w]]\|^{2}_{0,e}. \label{3.1.2}
\end{align}
\end{myLem}

%Based on this lemma, we have a discrete HDG Sobolev embedding
%relation  \cite{Chen2021}.
\begin{myLem} %[HDG Sobolev embedding]
 \label{L3.4}
	It holds
	\begin{equation}
	\|w\|_{0,\tilde{q},\mathcal{T}_{h}}
	\lesssim |\|(w,\mu)\||, \quad \forall (w,\mu)\in V_{h}\times\widehat{V}_{h}, \label{3.2}
	\end{equation}
	where $ 2\leq \tilde{q}<\infty $ when $ d=2 $ and  $ 2\leq \tilde{q}\leq6 $ when $ d=3 $.%, and $ C_{\tilde{q}} $ is a positive constant only depending on $ \tilde{q} $.
\end{myLem}
\begin{proof}
	For all $ (w,\mu)\in V_{h}\times\widehat{V}_{h} $, from Lemma \ref{L3.3} and the Poinc\'{a}re inequality, we have
	\begin{equation*}
	\|I_{h}w\|_{0,\tilde{q},\mathcal{T}_{h}}\lesssim\|I_{h}w\|_{1,2,\mathcal{T}_{h}}\lesssim\|\nabla I_{h}w\|_{0,\mathcal{T}_{h}}.
	\end{equation*}
	From (\ref{3.1.2}) and Lemma \ref{L3.2} it follows
	\begin{equation}
	\begin{aligned}
	\|\nabla I_{h}w\|_{0,\mathcal{T}_{h}}&\lesssim \|\nabla_{h} w\|_{0,\mathcal{T}_{h}}+(\sum_{e\in\varepsilon_{h}}h_{e}^{-1}\|[[w]]\|^{2}_{0,e})^{\frac{1}{2}}\\
	&\lesssim |\|(w,\mu)\||+(\sum_{e\in\varepsilon_{h}}h_{e}^{-1}\|[[w-\mu]]\|^{2}_{0,e})^{\frac{1}{2}}\\
	&\lesssim |\|(w,\mu)\||.  \label{3.1.3}
	\end{aligned}
    \end{equation}
    Using Lemma \ref{L3.3}, the inverse inequality and the properties of
    the projection-mean operator (\cite{Shi2013}) $ P_{h}: V_{h}\rightarrow W^{1,2}(\Omega)\cap W^{0,\tilde{q}}(\Omega)$, we obtain
    \begin{equation*}
    \begin{aligned}
    \|w-I_{h}w\|_{0,\tilde{q},\mathcal{T}_{h}}&\lesssim \|w-P_{h}w\|_{0,\tilde{q},\mathcal{T}_{h}}+\|P_{h}w-I_{h}w\|_{0,\tilde{q},\mathcal{T}_{h}}\\
    &\lesssim h^{1-\frac{d}{2}+\frac{d}{\tilde{q}}}\|\nabla_{h}w\|_{0,\mathcal{T}_{h}}+\|P_{h}w-I_{h}w\|_{1,2,\mathcal{T}_{h}}\\
    &\lesssim \|\nabla_{h}w\|_{0,\mathcal{T}_{h}}+\|w-P_{h}w\|_{1,2,\mathcal{T}_{h}}+\|w-I_{h}w\|_{1,2,\mathcal{T}_{h}}\\
    &\lesssim \|\nabla_{h}w\|_{0,\mathcal{T}_{h}}+\|\nabla_{h}(w-I_{h}w)\|_{0,\mathcal{T}_{h}}\\
    &\lesssim |\|(w,\mu)\||,  \label{3.1.4}
    \end{aligned}
    \end{equation*}
which,    together with (\ref{3.1.3}), %and (\ref{3.1.4}), 
yields the desired estimate (\ref{3.2}).
\end{proof}

%\subsection{Stability result}
\subsection{Well-posedness of the semi-discrete HDG scheme}
First, we have the following boundedness results for $ \mathcal{B}_{h} $.
\begin{myLem} \label{L3.6}
	For any $ (v_{h},\widehat{v}_{h}), (u_{h},\widehat{u}_{h}), (w,\mu)\in V_{h}\times\widehat{V}_{h} $, it holds
\begin{equation}
|\mathcal{B}_{h}(v_{h};u_{h},\widehat{u}_h;w,\mu)|\lesssim |\|(v_{h},\widehat{v}_h)\||\cdotp|\|(u_{h},\widehat{u}_h)\||\cdotp|\|(w,\mu)\||, \label{3.3}
\end{equation}
\begin{equation}
|\mathcal{B}_{h}(v_{h};u_{h},\widehat{u}_h;w,\mu)|\lesssim \|v_{h}\|_{0,3,\mathcal{T}_{h}}\cdotp|\|(u_{h},\widehat{u}_h)\||\cdotp|\|(w,\mu)\||. \label{3.3.1}
\end{equation}
\end{myLem}
\begin{proof}
For any $ (v_{h},\widehat{v}_{h}), (u_{h},\widehat{u}_{h}), (w,\mu)\in V_{h}\times\widehat{V}_{h} $, by the definition of $ \mathcal{B}_{h} $, we have
\begin{equation}
\begin{aligned}
3\mathcal{B}_{h}(v_{h};u_{h},\widehat{u}_h;w,\mu)&=\left[(\bm{b}(v_{h})\cdotp\nabla u_{h},w)_{\mathcal{T}_{h}}-(\bm{b}(v_{h})u_{h},\nabla w)_{\mathcal{T}_{h}}\right] \\
&\ \ +\left[ \left\langle \bm{b}(v_{h})\cdotp\bm{n}\widehat{u}_{h},w\right\rangle_{\partial\mathcal{T}_{h}}-\left\langle \bm{b}(v_{h})\cdotp\bm{n}u_{h},\mu\right\rangle_{\partial\mathcal{T}_{h}}\right] \\
&=:R_{1}+R_{2}. \label{3.4}
\end{aligned}
\end{equation}
Using the H\"{o}lder inequality and Lemma \ref{L3.4}, we have
\begin{equation}
\begin{aligned}
|R_{1}|&\lesssim \sum_{K\in\mathcal{T}_{h}}\|v_{h}\|_{0,3,K}\|\nabla u_{h}\|_{0,2,K}\|w\|_{0,6,K}+\sum_{K\in\mathcal{T}_{h}}\|v_{h}\|_{0,3,K}\|u_{h}\|_{0,6,K}\|\nabla w\|_{0,2,K}\\
&\lesssim \|v_{h}\|_{0,3,\mathcal{T}_{h}}\cdotp|\|(u_{h},\widehat{u}_h)\||\cdotp|\|(w,\mu)\||, \label{3.5}
\end{aligned}
\end{equation}
From triangle inequality we have
\begin{equation}
\begin{aligned}
|R_{2}|&\leq|\left\langle \bm{b}(v_{h})\cdotp\bm{n}(u_{h}-\widehat{u}_{h}),w\right\rangle_{\partial\mathcal{T}_{h}}|+|\left\langle \bm{b}(v_{h})\cdotp\bm{n}u_{h},w-\mu\right\rangle_{\partial\mathcal{T}_{h}}|\\
&=:T_{1}+T_{2}. \label{3.6}
\end{aligned}
\end{equation}
By the H\"{o}lder inequality, Lemma \ref{L3.1} and Lemma \ref{L3.4}, we obtain
\begin{equation}
\begin{aligned}
|T_{1}|&\lesssim \sum_{K\in\mathcal{T}_{h}}\|v_{h}\|_{0,3,\partial K}\|u_{h}-\widehat{u}_{h}\|_{0,2,\partial K}\|w\|_{0,6,\partial K}\\
&\lesssim \sum_{K\in\mathcal{T}_{h}}h_{T}^{-\frac{1}{3}}\|v_{h}\|_{0,3,K}\|u_{h}-\widehat{u}_{h}\|_{0,2,\partial K}h_{T}^{-\frac{1}{6}}\|w\|_{0,6,K}\\
&=\sum_{K\in\mathcal{T}_{h}}\|v_{h}\|_{0,3,K}h_{T}^{-\frac{1}{2}}\|u_{h}-\widehat{u}_{h}\|_{0,2,\partial K}\|w\|_{0,6,K}\\
&\lesssim \|v_{h}\|_{0,3,\mathcal{T}_{h}}\cdotp|\|(u_{h},\widehat{u}_h)\||\cdotp|\|(w,\mu)\||. \label{3.7}
\end{aligned}
\end{equation}
Similarly, we have
\begin{equation}
\begin{aligned}
|T_{2}|&\lesssim \sum_{K\in\mathcal{T}_{h}}\|v_{h}\|_{0,3,\partial K}\|u_{h}\|_{0,6,\partial K}\|w-\mu\|_{0,2,\partial K}\\
&\lesssim \sum_{K\in\mathcal{T}_{h}}\|v_{h}\|_{0,3,K}\|u_{h}\|_{0,6,K}h_{T}^{-\frac{1}{2}}\|w-\mu\|_{0,2,\partial K}\\
&\lesssim\|v_{h}\|_{0,3,\mathcal{T}_{h}}\cdotp|\|(u_{h},\widehat{u}_h)\||\cdotp|\|(w,\mu)\||. \label{3.8}
\end{aligned}
\end{equation}
As a result, the desired inequality (\ref{3.3.1}) follows from (\ref{3.4})-(\ref{3.8}). 

Since $$ \|v_{h}\|_{0,3,\mathcal{T}_{h}}\lesssim |\|(v_{h},\widehat{v}_h)\|| $$ by Lemma \ref{L3.4},   the inequality (\ref{3.3.1}) indicates  (\ref{3.3}).
\end{proof}

We also have the following stability result.
\begin{myTheo}
	For the numerical solution $ (\bm{q}_{h},u_{h},\widehat{u}_{h})\in \bm{Q}_{h}\times V_{h}\times\widehat{V}_{h}  $ to the scheme (\ref{2.12}) with initial setting $ u_{h}(0) $, it holds %the following stability result:
	\begin{equation}
	\|u_{h}(t)\|_{0,\mathcal{T}_{h}}^{2}+\nu\int_{0}^{t}|\|(u_{h},\widehat{u}_{h})\||^{2}d\tau\lesssim \|u_{h}(0)\|_{0,\mathcal{T}_{h}}^{2}+\frac{1}{\nu}\int_{0}^{t}\|f(\tau)\|_{0,\mathcal{T}_{h}}^{2}d\tau, \label{3.9}
	\end{equation}
	i.e., the numerical solution is stable with respect to initial approximate value and source term.
\end{myTheo}
\begin{proof}
	Taking $ (w,\mu)=(u_{h},\widehat{u}_{h}) $ in (\ref{2.12}), we get
	\begin{equation*}
	((u_{h})_{t},u_{h})_{\mathcal{T}_{h}}+\nu|\|(u_{h},\widehat{u}_{h})\||^{2}=(f,u_{h})_{\mathcal{T}_{h}},
	\end{equation*}
	then, using the Young’s inequality and Lemma \ref{L3.4}, we know that
	\begin{equation*}
	(f,u_{h})_{\mathcal{T}_{h}}\leq \frac{C}{2\nu}\|f\|^{2}_{0,\mathcal{T}_{h}}+\frac{\nu}{2}|\|(u_{h},\widehat{u}_{h})\||^{2},
	\end{equation*}
	where $C$ is a positive constant independent of $h$. Thus
	\begin{equation*}
	\dfrac{1}{2}\dfrac{d}{dt}\|u_{h}\|_{0,\mathcal{T}_{h}}^{2}+\frac{\nu}{2}|\|(u_{h},\widehat{u}_{h})\||^{2}\leq \frac{C}{\nu}\|f\|^{2}_{0,\mathcal{T}_{h}},
	\end{equation*}
	Integrating the above inequality with respect to $ t $ yields the desired result (\ref{3.9}).
\end{proof}

%\subsection{Existence and uniqueness of the semi-discrete solution}

We are now in a position to show   the global existence and uniqueness of the
semi-discrete solution   by the standard  theory of
 ordinary differential equations. 

\begin{myTheo}
	If $ f(\cdotp,t) $ is continuous with respect to $ t $, then the problem (\ref{2.12}) admits a unique solution $ (\bm{q}_{h},u_{h}(t),\widehat{u}_{h}(t)) $ for any $ t\in [0,T] $.
\end{myTheo}

\begin{proof}
	
%	Equation (\ref{2.12}) can be interpreted as a system of ordinary differential equations
%	(ODEs) for the coefficients in the expansion of $ (u_{h}(t),\widehat{u}_{h}(t)) $ in terms of basis functions
%	of the finite-dimensional
%	space $ V_{h} $ and $ \widehat{V}_{h} $.
	
	Let $ \bm{\varphi}, \bm{\phi} $ and $ \bm{\psi} $ be the bases of $ V_{h}|_{K} $, $ \widehat{V}_{h}|_{K} $ and $ [P_{k-1}(K)]^{d} $, respectively, with
	\begin{equation*}
	\bm{\varphi}=(\varphi_{1}, \ldots, \varphi_{m}),\ \ \bm{\phi}=(\phi_{1}, \ldots, \phi_{n}),\ \ \bm{\psi}=(\psi_{1}, \ldots, \psi_{p}).
	\end{equation*}
	Denote $ u_{h}(t)|_{K}:=\bm{\varphi}\bm{U}(t) $, $ \widehat{u}_{h}(t)|_{\partial K}:=\bm{\phi}\widehat{\bm{U}}(t) $ and $ -\bm{q}_{h}(t)|_{K}:=\bm{\psi}\bm{Q}(t) $
	with
	\begin{equation*}
	\bm{U}(t)=(U_{1}(t), \ldots, U_{m}(t))^{T},\ \ \widehat{\bm{U}}(t)=(\widehat{U}_{1}(t), \ldots, \widehat{U}_{n}(t))^{T},\ \ \bm{Q}(t)=(Q_{1}(t), \ldots, Q_{p}(t))^{T}.
	\end{equation*}
	Set
	
	\begin{align*}
	\mathcal{M}_{1}&=\sum_{K\in\mathcal{T}_{h}}\int_{K}\bm{\psi}^{T}\bm{\psi}dx, &\ \mathcal{M}_{2}&=\sum_{K\in\mathcal{T}_{h}}\int_{K}
	(\nabla\cdotp\bm{\psi})^{T}\bm{\varphi}dx, &\ \mathcal{M}_{3}&=\sum_{K\in\mathcal{T}_{h}}\int_{\partial K}(\bm{\psi}\bm{n})^{T}\bm{\phi}ds,\\
	\mathcal{M}_{4}&=\sum_{K\in\mathcal{T}_{h}}\int_{K}\bm{\varphi}^{T}\bm{\varphi}dx, &\ \mathcal{M}_{5}&=\sum_{K\in\mathcal{T}_{h}}\int_{\partial K}\tau\bm{\varphi}^{T}\bm{\varphi}ds, &\ \mathcal{M}_{6}&=\sum_{K\in\mathcal{T}_{h}}\int_{\partial K}\tau(\bm{\varphi})^{T}\bm{\phi}ds,\\
	\mathcal{M}_{7}&=\sum_{K\in\mathcal{T}_{h}}\int_{\partial K}\tau(\bm{\phi})^{T}\bm{\phi}ds, &\ \mathcal{M}_{8}(\tilde{\bm{U}})&=\sum_{K\in\mathcal{T}_{h}}\int_{K}\frac{1}{3}(\nabla\bm{\varphi})^{T}\bm{b}(\bm{\varphi}\tilde{\bm{U}})\bm{\varphi}dx, &\
	\mathcal{M}_{9}(\tilde{\bm{U}})&=\sum_{K\in\mathcal{T}_{h}}\int_{\partial K}\frac{1}{3}(\bm{b}(\bm{\varphi}\tilde{\bm{U}})\bm{n})\bm{\varphi}^{T}\bm{\phi}ds,\\
    \mathcal{F}(t)&=\sum_{K\in\mathcal{T}_{h}}\int_{K}\bm{\varphi}^{T}f(t)dx.
	\end{align*}
	Here we denote that $ \tilde{\bm{Q}}(t)|_{K}=\bm{Q}(t) $, $ \tilde{\bm{U}}(t)|_{K}=\bm{U}(t) $ and $ \tilde{\widehat{\bm{U}}}(t)|_{\partial K}=\widehat{\bm{U}}(t) $. Thus, the system (\ref{2.12}) can be written as
	\begin{equation}
	\left\{
	\begin{aligned}
	&\mathcal{M}_{1}\tilde{\bm{Q}}(t)+\mathcal{M}_{2}\tilde{\bm{U}}(t)-\mathcal{M}_{3}\tilde{\widehat{\bm{U}}}(t)=0,\\
	&\mathcal{M}_{4}\frac{d\tilde{\bm{U}}(t)}{dt}-\nu\mathcal{M}_{2}^{T}\tilde{\bm{Q}}(t)+(\nu\mathcal{M}_{5}-\mathcal{M}_{8}(\tilde{\bm{U}}(t))+\mathcal{M}_{8}^{T}(\tilde{\bm{U}}(t)))\tilde{\bm{U}}(t)+(\mathcal{M}_{9}(\tilde{\bm{U}}(t))-\nu\mathcal{M}_{6})\tilde{\widehat{\bm{U}}}(t)=\mathcal{F}(t),\\
	&\nu\mathcal{M}_{3}^{T}\tilde{\bm{Q}}(t)-(\nu\mathcal{M}_{6}^{T}+\mathcal{M}_{9}^{T}(\tilde{\bm{U}}(t)))\tilde{\bm{U}}(t)+\nu\mathcal{M}_{7}\tilde{\widehat{\bm{U}}}(t)=0. \label{3.11}
	\end{aligned}
	\right.
	\end{equation}
	Since $ \mathcal{M}_{1}, \mathcal{M}_{4} $ and $ \mathcal{M}_{7} $ are symmetric positive defined, we can eliminate $ \tilde{\bm{Q}}(t) $ and $ \tilde{\widehat{\bm{U}}}(t) $ in (\ref{3.11}) to get
	\begin{equation}
	\mathcal{M}_{4}\frac{d\tilde{\bm{U}}(t)}{dt}+\tilde{\mathcal{M}}(\tilde{\bm{U}}(t))\tilde{\bm{U}}(t)=\mathcal{F}(t), \label{3.12}
	\end{equation}
	where
	\begin{align*}
	\tilde{\mathcal{M}}(\tilde{\bm{U}}(t))=&\nu\mathcal{M}_{2}^{T}\mathcal{M}_{1}^{-1}\mathcal{M}_{2}+\nu\mathcal{M}_{5}-\mathcal{M}_{8}(\tilde{\bm{U}})+\mathcal{M}_{8}^{T}(\tilde{\bm{U}})
	-(\nu\mathcal{M}_{2}^{T}\mathcal{M}_{1}^{-1}\mathcal{M}_{3}-\mathcal{M}_{9}(\tilde{\bm{U}})+\nu\mathcal{M}_{6})\mathcal{M}^{*}
	\end{align*}
	and
	\begin{equation*}
	\mathcal{M}^{*}=(\mathcal{M}_{7}+\mathcal{M}_{3}^{T}\mathcal{M}_{1}^{-1}\mathcal{M}_{3})^{-1}(\mathcal{M}_{3}^{T}\mathcal{M}_{1}^{-1}\mathcal{M}_{2}+\mathcal{M}^{T}_{6}+\frac{1}{\nu}\mathcal{M}_{9}^{T}(\tilde{\bm{U}})).
	\end{equation*}
	According to the stability result (\ref{3.9}), we know that $ \mathcal{M}_{8}(\tilde{\bm{U}}(t)) $ and $ \mathcal{M}_{9}(\tilde{\bm{U}}(t)) $ are bounded in $ [0,T] $, which implies that $ \tilde{\mathcal{M}}(\tilde{\bm{U}}(t))\tilde{\bm{U}}(t) $ is globally Lipschitz continuous with respect to $ \tilde{\bm{U}}(t) $. In addition, $ \mathcal{M}_{4} $ is symmetric positive defined, by the standard ODE theory \cite{Warga1972}, there exist a unique solution to (\ref{3.12}) on $ [0,T] $, which
	means the existence and uniqueness of $ (\bm{q}_{h}(t),u_{h}(t),\widehat{u}_{h}(t)) $ in (\ref{2.12}) on the interval $ [0,T] $. 
\end{proof}

%\section{Error estimation for the semi-discrete HDG scheme}

\subsection{A priori error estimation}
This section is devoted to the error estimation of the HDG scheme (\ref{2.12}). 
For the $ L^{2} $ projections
$%\begin{equation*}
(\boldsymbol{\Pi}_{k-1}^{o}\bm{q}(t),\Pi_{k}^{o}u(t),\Pi_{l}^{\partial}u(t)),
$%\end{equation*}
we have the following standard estimates \cite{Brenner1994}:
\begin{equation*}
\begin{aligned}
\|\bm{q}-\boldsymbol{\Pi}_{k-1}^{o}\bm{q}\|_{0,\mathcal{T}_{h}}&\leq Ch^{k}\|\bm{q}\|_{k,\Omega},\ \forall \bm{q}\in [H^{k}(\Omega)]^{d},\\
\|u-\Pi_{k}^{o}u\|_{0,\mathcal{T}_{h}}&\leq Ch^{k+1}\|u\|_{k+1,\Omega},\ \forall u\in H^{k+1}(\Omega),\\
\|u-\Pi_{k}^{o}u\|_{0,\partial\mathcal{T}_{h}}&\leq Ch^{k+\frac{1}{2}}\|u\|_{k+1,\Omega},\ \forall u\in H^{k+1}(\Omega),\\
\|w\|_{0,\partial\mathcal{T}_{h}}&\leq Ch^{-\frac{1}{2}}\|w\|_{0,\mathcal{T}_{h}},\ \forall w\in V_{h}.
\end{aligned}
\end{equation*}

We also need the Gronwall inequality:
\begin{myLem} \label{L3.5}
	%\cite{Teschl2012} 
	Suppose $ \rho(t)\geq 0 $ satisfies 
	\begin{equation*}
	\rho(t)\leq \alpha+\int_{0}^{t}\beta(s)\rho(s)ds,
	\end{equation*}
	with $ \alpha, \beta(s)\geq 0 $. Then it follows
	\begin{equation*}
	\rho(t)\leq \alpha\exp(\int_{0}^{t}\beta(s)ds).
	\end{equation*}
\end{myLem}

\begin{myLem} \label{L4.1}
For any $ u\in H_{0}^{1}(\Omega) $, $ (w,\mu)\in V_{h}\times\widehat{V}_{h} $, it holds
\begin{equation}
\mathcal{B}_{h}(\Pi_{k}^{o}u(t);\Pi_{k}^{o}u(t),\Pi_{l}^{\partial}u(t);w,\mu)=(\bm{b}(u)\cdot\nabla u,w)+E_{N}(u;u,w,\mu), \label{4.1}
\end{equation}
where
\begin{equation*}
\begin{aligned}
E_{N}(u;u,w,\mu):=&\frac{1}{3}(\bm{b}(u-\Pi_{k}^{o}u)u,\nabla w)_{\mathcal{T}_{h}}+\frac{1}{3}(\bm{b}(\Pi_{k}^{o}u)(u-\Pi_{k}^{o}u),\nabla w)_{\mathcal{T}_{h}}\\
&-\frac{1}{3}(\bm{b}(u-\Pi_{k}^{o}u)\cdotp\nabla\Pi_{k}^{o}u,w)_{\mathcal{T}_{h}}-\frac{1}{3}(\bm{b}(u)\cdotp(\nabla u-\nabla\Pi_{k}^{o}u),w)_{\mathcal{T}_{h}}\\
&+\frac{1}{3}\left\langle \bm{b}(u-\Pi_{k}^{o}u)\cdotp\bm{n},\mu\Pi_{k}^{o}u\right\rangle_{\partial\mathcal{T}_{h}}+\frac{1}{3}\left\langle \bm{b}(u)\cdotp\bm{n},\mu(u-\Pi_{k}^{o}u)\right\rangle_{\partial\mathcal{T}_{h}}\\
&-\frac{1}{3}\left\langle \bm{b}(u-\Pi_{k}^{o}u)\cdotp\bm{n},w\Pi_{l}^{\partial}u\right\rangle_{\partial\mathcal{T}_{h}}-\frac{1}{3}\left\langle \bm{b}(u)\cdotp\bm{n},w(u-\Pi_{l}^{\partial}u)\right\rangle_{\partial\mathcal{T}_{h}}.
\end{aligned}
\end{equation*}
\end{myLem}
\begin{proof}
For any $ w\in V_{h} $, by $ \left\langle \bm{b}(u)\cdotp\bm{n},u\mu \right\rangle_{\partial\mathcal{T}_{h}}=0  $ we have
\begin{equation*}
\begin{aligned}
(\bm{b}(u)\cdot\nabla u,w)_{\mathcal{T}_{h}}&=\frac{1}{3}(\nabla\cdotp\bm{b}(u),uw)_{\mathcal{T}_{h}}+\frac{2}{3}(\bm{b}(u),w\nabla u)_{\mathcal{T}_{h}}\\
&=\frac{1}{3}\left\langle \bm{b}(u)\cdotp\bm{n},uw\right\rangle_{\partial\mathcal{T}_{h}}-\frac{1}{3}(\bm{b}(u),w\nabla u)_{\mathcal{T}_{h}}
 -\frac{1}{3}(\bm{b}(u),u\nabla w)_{\mathcal{T}_{h}}+\frac{2}{3}(\bm{b}(u),w\nabla u)_{\mathcal{T}_{h}}\\
&=\frac{1}{3}(\bm{b}(u),w\nabla u)_{\mathcal{T}_{h}}-\frac{1}{3}(\bm{b}(u)u,\nabla w)_{\mathcal{T}_{h}}
 +\frac{1}{3}\left\langle \bm{b}(u)\cdotp\bm{n},u(w-\mu) \right\rangle_{\partial\mathcal{T}_{h}}
\end{aligned}
\end{equation*}
and
\begin{equation*}
\begin{aligned}
&\mathcal{B}_{h}(\Pi_{k}^{o}u(t);\Pi_{k}^{o}u(t),\Pi_{l}^{\partial}u(t);w,\mu)\\
=&-\frac{1}{3}(\bm{b}(\Pi_{k}^{o}u)\Pi_{k}^{o}u,\nabla w)_{\mathcal{T}_{h}}+\frac{1}{3}(\bm{b}(\Pi_{k}^{o}u)\cdotp\nabla\Pi_{k}^{o}u,w)_{\mathcal{T}_{h}}
-\frac{1}{3}\left\langle \bm{b}(\Pi_{k}^{o}u)\cdotp\bm{n},\mu\Pi_{k}^{o}u\right\rangle_{\partial\mathcal{T}_{h}}+ \frac{1}{3}\left\langle \bm{b}(\Pi_{k}^{o}u)\cdotp\bm{n},w\Pi_{l}^{\partial}u\right\rangle_{\partial\mathcal{T}_{h}}, 
\end{aligned}
\end{equation*}
which imply
\begin{equation*}
\begin{aligned}
&-\frac{1}{3}(\bm{b}(\Pi_{k}^{o}u)\Pi_{k}^{o}u,\nabla w)_{\mathcal{T}_{h}}+\frac{1}{3}(\bm{b}(\Pi_{k}^{o}u)\cdotp\nabla\Pi_{k}^{o}u,w)_{\mathcal{T}_{h}}\\
=&\frac{1}{3}(\bm{b}(u-\Pi_{k}^{o}u)u,\nabla w)_{\mathcal{T}_{h}}+\frac{1}{3}(\bm{b}(\Pi_{k}^{o}u)(u-\Pi_{k}^{o}u),\nabla w)_{\mathcal{T}_{h}}-\frac{1}{3}(\bm{b}(u)u,\nabla w)_{\mathcal{T}_{h}}\\
&-\frac{1}{3}(\bm{b}(u-\Pi_{k}^{o}u)\cdotp\nabla\Pi_{k}^{o}u,w)_{\mathcal{T}_{h}}-\frac{1}{3}(\bm{b}(u)\cdotp(\nabla u-\nabla\Pi_{k}^{o}u),w)_{\mathcal{T}_{h}}+\frac{1}{3}(\bm{b}(u),w\nabla u)_{\mathcal{T}_{h}}, 
\end{aligned}
\end{equation*}
and
\begin{equation*}
\begin{aligned}
&-\frac{1}{3}\left\langle \bm{b}(\Pi_{k}^{o}u)\cdotp\bm{n},\mu\Pi_{k}^{o}u\right\rangle_{\partial\mathcal{T}_{h}}+ \frac{1}{3}\left\langle \bm{b}(\Pi_{k}^{o}u)\cdotp\bm{n},w\Pi_{l}^{\partial}u\right\rangle_{\partial\mathcal{T}_{h}}\\
=&\frac{1}{3}\left\langle \bm{b}(u-\Pi_{k}^{o}u)\cdotp\bm{n},\mu\Pi_{k}^{o}u\right\rangle_{\partial\mathcal{T}_{h}}+\frac{1}{3}\left\langle \bm{b}(u)\cdotp\bm{n},\mu(u-\Pi_{k}^{o}u)\right\rangle_{\partial\mathcal{T}_{h}}-\frac{1}{3}\left\langle \bm{b}(u)\cdotp\bm{n},\mu u\right\rangle_{\partial\mathcal{T}_{h}}\\
&-\frac{1}{3}\left\langle \bm{b}(u-\Pi_{k}^{o}u)\cdotp\bm{n},w\Pi_{l}^{\partial}u\right\rangle_{\partial\mathcal{T}_{h}}-\frac{1}{3}\left\langle \bm{b}(u)\cdotp\bm{n},w(u-\Pi_{l}^{\partial}u)\right\rangle_{\partial\mathcal{T}_{h}}+\frac{1}{3}\left\langle \bm{b}(u)\cdotp\bm{n},wu\right\rangle_{\partial\mathcal{T}_{h}}. 
\end{aligned}
\end{equation*}
Combining the above four equalities gives (\ref{4.1}).
\end{proof}
\begin{myLem} \label{L4.2}
For $ u\in H^{k+1}(\Omega) $, it holds
\begin{equation}
|E_{N}(u;u,w,\mu)|\lesssim h^{k}\|u\|_{k+1}\|u\|_{2}|\|(w,\mu)\||,\ \ \forall (w,\mu)\in V_{h}\times\widehat{V}_{h}. \label{4.2}
\end{equation}
\end{myLem}
\begin{proof}
From the H\"{o}lder inequality, the sobolev inequality, and the projection properties, we have the following estimates:
\begin{equation*}
\begin{aligned}
|(\bm{b}(u-\Pi_{k}^{o}u)u,\nabla w)_{\mathcal{T}_{h}}|&\lesssim \sum_{K\in\mathcal{T}_{h}}|u-\Pi_{k}^{o}u|_{0,2,K}|u|_{0,\infty,K}\|\nabla w\|_{0,2,K}\\
&\lesssim h^{k+1}\|u\|_{k+1}\|u\|_{2}\|\nabla w\|_{0,\mathcal{T}_{h}}\\
&\lesssim h^{k+1}\|u\|_{k+1}\|u\|_{2}|\|(w,\mu)\||,
\end{aligned}\end{equation*}

\begin{equation*}
\begin{aligned}
|(\bm{b}(\Pi_{k}^{o}u)(u-\Pi_{k}^{o}u),\nabla w)_{\mathcal{T}_{h}}|&\leq C\sum_{K\in\mathcal{T}_{h}}|\Pi_{k}^{o}u|_{0,\infty,K}\|u-\Pi_{k}^{o}u\|_{0,2,K}\|\nabla w\|_{0,2,K}\\
&\lesssim \sum_{K\in\mathcal{T}_{h}}|u|_{0,\infty}h^{k+1,K}|u|_{k+1,K}\|\nabla w\|_{0,2,K}\\
&\lesssim h^{k+1}\|u\|_{k+1}\|u\|_{2}\|\nabla w\|_{0,\mathcal{T}_{h}}\\
&\lesssim h^{k+1}\|u\|_{k+1}\|u\|_{2}|\|(w,\mu)\||,
\end{aligned}\end{equation*}

\begin{equation*}
\begin{aligned}
|(\bm{b}(u-\Pi_{k}^{o}u)\cdotp\nabla\Pi_{k}^{o}u,w)_{\mathcal{T}_{h}}|&\lesssim \sum_{K\in\mathcal{T}_{h}}|u-\Pi_{k}^{o}u|_{0,3,K}|\nabla\Pi_{k}^{o}u|_{0,2,K}\|w\|_{0,6,K}\\
&\lesssim \sum_{K\in\mathcal{T}_{h}}h^{k+1-\frac{d}{6}}|u|_{k+1,K}|u|_{1,2,K}\|w\|_{0,6,K}\\
&\lesssim h^{k+1-\frac{d}{6}}\|u\|_{k+1}\|u\|_{2}\|w\|_{0,6,\mathcal{T}_{h}},\\
&\lesssim h^{k+1-\frac{d}{6}}\|u\|_{k+1}\|u\|_{2}|\|(w,\mu)\||,
\end{aligned}\end{equation*}

\begin{equation*}
\begin{aligned}
|(\bm{b}(u)\cdotp(\nabla u-\nabla\Pi_{k}^{o}u),w)_{\mathcal{T}_{h}}|&\lesssim \sum_{K\in\mathcal{T}_{h}}|u|_{0,\infty,K}|\nabla u-\nabla\Pi_{k}^{o}u|_{0,2,K}\|w\|_{0,2,K}\\
&\lesssim \sum_{K\in\mathcal{T}_{h}}|u|_{0,\infty,K}h^{k}|u|_{k+1,K}\|w\|_{0,2,K}\\
&\lesssim h^{k}\|u\|_{k+1}\|u\|_{2}\|w\|_{0,\mathcal{T}_{h}},\\
&\lesssim h^{k}\|u\|_{k+1}\|u\|_{2}|\|(w,\mu)\||,
\end{aligned}\end{equation*}

\begin{equation*}
\begin{aligned}
&\ \ \ \ |\left\langle \bm{b}(u-\Pi_{k}^{o}u)\cdotp\bm{n},\mu\Pi_{k}^{o}u\right\rangle_{\partial\mathcal{T}_{h}}|\\
&=|\left\langle \bm{b}(u-\Pi_{k}^{o}u)\cdotp\bm{n},(w-\mu)\Pi_{k}^{o}u\right\rangle_{\partial\mathcal{T}_{h}}-\left\langle \bm{b}(u-\Pi_{k}^{o}u)\cdotp\bm{n},w\Pi_{k}^{o}u\right\rangle_{\partial\mathcal{T}_{h}}|\\
&\lesssim \sum_{K\in\mathcal{T}_{h}}(|u-\Pi_{k}^{o}u|_{0,2,\partial K}|w-\mu|_{0,2,\partial K}|\Pi_{k}^{o}u|_{0,\infty,\partial K}+|u-\Pi_{k}^{o}u|_{0,2,\partial K}|w|_{0,2,\partial K}|\Pi_{k}^{o}u|_{0,\infty,\partial K})\\
&\lesssim \sum_{K\in\mathcal{T}_{h}}(h^{k+\frac{1}{2}}|u|_{k+1,K}|w-\mu|_{0,2,\partial K}|u|_{0,\infty,K}+h^{k+\frac{1}{2}}|u|_{k+1,K}h^{-\frac{1}{2}}|w|_{0,2,K}|u|_{0,\infty,K})\\
&\lesssim h^{k}\|u\|_{k+1}\|u\|_{2}|\|(w,\mu)\||,\\
\\
&\ \ \ \ |\left\langle \bm{b}(u)\cdotp\bm{n},\mu(u-\Pi_{k}^{o}u)\right\rangle_{\partial\mathcal{T}_{h}}|\\
&=|\left\langle \bm{b}(u)\cdotp\bm{n},(w-\mu)(u-\Pi_{k}^{o}u)\right\rangle_{\partial\mathcal{T}_{h}}-|\left\langle \bm{b}(u)\cdotp\bm{n},w(u-\Pi_{k}^{o}u)\right\rangle_{\partial\mathcal{T}_{h}}||\\
&\lesssim \sum_{K\in\mathcal{T}_{h}}(|u|_{0,\infty,\partial K}|w-\mu|_{0,2,\partial K}|u-\Pi_{k}^{o}u|_{0,2,\partial K}+|u|_{0,\infty,\partial K}|w|_{0,2,\partial K}|u-\Pi_{k}^{o}u|_{0,2,\partial K})\\
&\lesssim \sum_{K\in\mathcal{T}_{h}}(|u|_{0,\infty,K}|w-\mu|_{0,2,\partial K}h^{k+\frac{1}{2}}|u|_{k+1,K}+|u|_{0,\infty,K}h^{-\frac{1}{2}}|w|_{0,2,K}h^{k+\frac{1}{2}}|u|_{k+1,K})\\
&\lesssim h^{k}\|u\|_{k+1}\|u\|_{2}|\|(w,\mu)\||,
\end{aligned}
\end{equation*}

\begin{equation*}
\begin{aligned}
&\ \ \ \ |\left\langle \bm{b}(u-\Pi_{k}^{o}u)\cdotp\bm{n},w\Pi_{l}^{\partial}u\right\rangle_{\partial\mathcal{T}_{h}}|\\
&\lesssim \sum_{K\in\mathcal{T}_{h}}|u-\Pi_{k}^{o}u|_{0,2,\partial K}|w|_{0,2,\partial K}|\Pi_{l}^{\partial}u|_{0,\infty,\partial K}\\
&\lesssim \sum_{K\in\mathcal{T}_{h}}h^{k+\frac{1}{2}}|u|_{k+1,K}h^{-\frac{1}{2}}|w|_{0,2,K}|u|_{0,\infty,K}\\
&\lesssim h^{k}\|u\|_{k+1}\|u\|_{2}|\|(w,\mu)\||,\\
\\
&\ \ \ \ |\left\langle \bm{b}(u)\cdotp\bm{n},w(u-\Pi_{l}^{\partial}u)\right\rangle_{\partial\mathcal{T}_{h}}|=|\left\langle \bm{b}(u)\cdotp\bm{n}(u-\Pi_{l}^{\partial}u),w-\mu\right\rangle_{\partial\mathcal{T}_{h}}|\\
&\lesssim \sum_{K\in\mathcal{T}_{h}}|u|_{0,\infty,\partial K}|u-\Pi_{l}^{\partial}u|_{0,2,\partial K}|w-\mu|_{0,2,\partial K}\\
&\lesssim \sum_{K\in\mathcal{T}_{h}}|u|_{0,\infty,K}h^{l+\frac{1}{2}}|u|_{l+1,K}h^{\frac{1}{2}}(h^{-\frac{1}{2}}|w-\mu|_{0,2,\partial K})\\
&\lesssim h^{k}\|u\|_{k+1}\|u\|_{2}|\|(w,\mu)\||.
\end{aligned}
\end{equation*}
As a result, the desired result follows from the definition of $ E_{N}(u;u,w,\mu) $. 
\end{proof}
\begin{myLem} \label{L4.3}
Let $ (u,q) $ be the solution to the problem (\ref{2.5}), then for any $ (w,\mu)\in V_{h}\times\widehat{V}_{h} $ it holds
\begin{equation}
\left\{
\begin{aligned}
\boldsymbol{\Pi}_{k-1}^{o}\bm{q}=&-\mathcal{K}_{h}(\Pi_{k}^{o}u,\Pi_{l}^{\partial}u),\\
((\Pi_{k}^{o}u)_{t},w)_{\mathcal{T}_{h}}&+\nu(\mathcal{K}_{h}(\Pi_{k}^{o}u,\Pi_{l}^{\partial}u),\mathcal{K}_{h}(w,\mu))_{\mathcal{T}_{h}}
+\mathcal{S}_{h}(\Pi_{k}^{o}u,\Pi_{l}^{\partial}u;w,\mu)+\mathcal{B}_{h}(\Pi_{k}^{o}u;\Pi_{k}^{o}u,\Pi_{l}^{\partial}u;w,\mu)\\
=&((\Pi_{k}^{o}u-u)_{t},w)_{\mathcal{T}_{h}}+(f,w)+E_{L}(u,w,\mu)+E_{N}(u;u,w,\mu), \label{4.3}
\end{aligned}
\right.
\end{equation}
%\begin{equation}
%\begin{aligned}
%((\Pi_{k}^{o}u)_{t}&,w)_{\mathcal{T}_{h}}+\mathcal{A}(\boldsymbol{\Pi}_{k-1}^{o}\bm{q},\Pi_{k}^{o}u,\Pi_{l}^{\partial}u;\bm{r},w,\mu)+\mathcal{B}(\Pi_{k}^{o}u;\Pi_{k}^{o}u,\Pi_{l}^{\partial}u;w,\mu)\\
%&=((\Pi_{k}^{o}u-u)_{t},w)_{\mathcal{T}_{h}}+(f,w)+E_{L}(\bm{q},u,\bm{r},w,\mu)+E_{N}(u;u,w,\mu),
%\end{aligned}
%\end{equation}
where
\begin{equation*}
\begin{aligned}
E_{L}(u,w,\mu):=&-\nu\left\langle (\nabla u-\boldsymbol{\Pi}_{k-1}^{o}\nabla u)\cdotp\bm{n},w-\mu\right\rangle_{\partial\mathcal{T}_{h}}
+\nu\left\langle\tau(\Pi_{k}^{o}u-u),\Pi_{l}^{\partial}w-\mu \right\rangle_{\partial\mathcal{T}_{h}}.
\end{aligned}
\end{equation*}
\end{myLem}
\begin{proof}
By the definition of $ \mathcal{K}_{h} $, we   get
\begin{equation*}
\begin{aligned}
&((\Pi_{k}^{o}u)_{t},w)_{\mathcal{T}_{h}}+\nu(\mathcal{K}_{h}(\Pi_{k}^{o}u,\Pi_{l}^{\partial}u),\mathcal{K}_{h}(w,\mu))_{\mathcal{T}_{h}} +\mathcal{S}_{h}(\Pi_{k}^{o}u,\Pi_{l}^{\partial}u;w,\mu)+\mathcal{B}_{h}(\Pi_{k}^{o}u;\Pi_{k}^{o}u,\Pi_{l}^{\partial}u;w,\mu)\\
=&((\Pi_{k}^{o}u)_{t},w)_{\mathcal{T}_{h}}
+\left(\nu(\nabla\cdotp \boldsymbol{\Pi}_{k-1}^{o}\bm{q},w)_{\mathcal{T}_{h}}-\nu\left\langle \boldsymbol{\Pi}_{k-1}^{o}\bm{q}\cdotp\bm{n},\mu\right\rangle_{\partial\mathcal{T}_{h}}\right)
+\nu\left\langle\tau(\Pi_{l}^{\partial}\Pi_{k}^{o}u-\Pi_{l}^{\partial}u),\Pi_{l}^{\partial}w-\mu \right\rangle_{\partial\mathcal{T}_{h}}\\   
&+\mathcal{B}_{h}(\Pi_{k}^{o}u;\Pi_{k}^{o}u,\Pi_{l}^{\partial}u;w,\mu)\\
:=&\sum_{i=1}^{4}R_{i}.
\end{aligned}
\end{equation*}
From the Green’s formula, the
properties of the projections $ \boldsymbol{\Pi}_{k-1}^{o}, \Pi_{k}^{o} $ and $ \Pi_{l}^{\partial} $, it follows
\begin{equation*}
\begin{aligned}
R_{1}&=((\Pi_{k}^{o}u-u)_{t},w)_{\mathcal{T}_{h}}+(u_{t},w),\\
R_{2}&=\nu(\nabla\cdotp \boldsymbol{\Pi}_{k-1}^{o}\bm{q},w)_{\mathcal{T}_{h}}-\nu\left\langle \boldsymbol{\Pi}_{k-1}^{o}\bm{q}\cdotp\textbf{n},\mu\right\rangle_{\partial\mathcal{T}_{h}}\\
&=-\nu(\nabla\cdotp\bm{q}- \nabla\cdotp \boldsymbol{\Pi}_{k-1}^{o}\bm{q},w)_{\mathcal{T}_{h}}+\nu(\nabla\cdotp\bm{q},w)_{\mathcal{T}_{h}}   -\nu\left\langle (\bm{q}-\boldsymbol{\Pi}_{k-1}^{o}\bm{q})\cdotp\bm{n},w-\mu\right\rangle_{\partial\mathcal{T}_{h}}\\
&\quad +\nu\left\langle (\bm{q}-\boldsymbol{\Pi}_{k-1}^{o}\bm{q})\cdotp\bm{n},w\right\rangle_{\partial\mathcal{T}_{h}}
  -\nu\left\langle (\bm{q}\cdotp\bm{n},\mu\right\rangle_{\partial\mathcal{T}_{h}}\\
&=\nu(\bm{q}-\boldsymbol{\Pi}_{k-1}^{o}\bm{q},\nabla w)_{\mathcal{T}_{h}}-\nu\left\langle (\bm{q}-\boldsymbol{\Pi}_{k-1}^{o}\bm{q})\cdotp\bm{n},w-\mu\right\rangle_{\partial\mathcal{T}_{h}} +\nu(\nabla\cdotp\bm{q}, w)\\
&=-\nu\left\langle (\nabla u-\boldsymbol{\Pi}_{k-1}^{o}\nabla u)\cdotp\bm{n},w-\mu\right\rangle_{\partial\mathcal{T}_{h}}-\nu(\Delta u, w),\\
R_{3}&=\nu\left\langle\tau(\Pi_{l}^{\partial}\Pi_{k}^{o}u-\Pi_{l}^{\partial}u),\Pi_{l}^{\partial}w-\mu \right\rangle_{\partial\mathcal{T}_{h}} =\nu\left\langle\tau(\Pi_{k}^{o}u-u),\Pi_{l}^{\partial}w-\mu \right\rangle_{\partial\mathcal{T}_{h}},
\end{aligned}
\end{equation*}
and, by (\ref{4.1}) we have
\begin{equation*}
R_{4}=(\bm{b}(u)\cdot\nabla u,w)+E_{N}(u;u,w,\mu).
\end{equation*}
Finally, combining the above equations implies the desired relation (\ref{4.3}).
\end{proof}
\begin{myLem} \label{L4.4}
For $ u\in H^{k+1}(\Omega) $, it holds
\begin{equation}
|E_{L}(u,w,\mu)|\lesssim \nu h^{k}\|u\|_{k+1}|\|(w,\mu)\||,\ \ \forall (w,\mu)\in V_{h}\times\widehat{V}_{h}. \label{4.4}
\end{equation}
\end{myLem}
\begin{proof} The desired conclusion follows from 
\begin{eqnarray*}
%\begin{aligned}
  |\left\langle (\nabla u-\boldsymbol{\Pi}_{k-1}^{o}\nabla u)\cdotp\bm{n},w-\mu\right\rangle_{\partial\mathcal{T}_{h}}|
&\lesssim& \sum_{K\in\mathcal{T}_{h}}|\nabla u-\boldsymbol{\Pi}_{k-1}^{o}\nabla u|_{0,2,\partial K}h^{\frac{1}{2}}(h^{-\frac{1}{2}}|w-\mu|_{0,2,\partial K})\\
&\lesssim& h^{k}\|\nabla u\|_{k}|\|(w,\mu)\||,\\
  |\left\langle\tau(\Pi_{k}^{o}u-u),\Pi_{l}^{\partial}w-\mu \right\rangle_{\partial\mathcal{T}_{h}}|
&\lesssim& \sum_{K\in\mathcal{T}_{h}}\tau^{\frac{1}{2}}|\Pi_{k}^{o}u-u|_{0,2,\partial K}\tau^{\frac{1}{2}}|\Pi_{l}^{\partial}w-\mu|_{0,2,\partial K}\\
&\lesssim& h^{k}\|u\|_{k+1}|\|(w,\mu)\||.
%\end{aligned}
\end{eqnarray*}

%i.e. (\ref{4.4}) holds. 
\end{proof}

%
%To bound the error between the solution of the HDG formulation (\ref{2.12}) and the system (\ref{2.5}). We first
%derive the error equation summarized in the next lemma. 
Set
\begin{equation*}
\begin{aligned}
\xi_{h}^{\bm{q}}&:=\boldsymbol{\Pi}_{k-1}^{o}\bm{q}-\bm{q}_{h},\ \ \xi_{h}^{u}:=\Pi_{k}^{o}u-u_{h},\ \ \xi_{h}^{\widehat{u}}:=\Pi_{l}^{\partial}u-\widehat{u}_{h},\\
\eta_{h}^{\bm{q}}&:=\bm{q}-\boldsymbol{\Pi}_{k-1}^{o}\bm{q},\ \ \ \  \eta_{h}^{u}:=u-\Pi_{k}^{o}u,\ \ \  \eta_{h}^{\widehat{u}}:=u-\Pi_{l}^{\partial}u.
\end{aligned}
\end{equation*}
By substracting (\ref{2.12}) from (\ref{4.3}), we can obtain the following error equations.
\begin{myLem}[] 
	%(Error equation)
	Let $ (u,q) $ be the solution to the problem (\ref{2.5}) and $ (\bm{q}_{h},u_{h},\widehat{u}_{h})\in  \bm{Q}_{h}\times V_{h}\times\widehat{V}_{h} $ be the solution of (\ref{2.12}). Then, for any $ (w,\mu)\in V_{h}\times\widehat{V}_{h} $, we have
	\begin{equation}
	\left\{
	\begin{aligned}
	\xi_{h}^{\bm{q}}+\mathcal{K}_{h}(\xi_{h}^{u},\xi_{h}^{\widehat{u}})&=0,\\
	((\xi_{h}^{u})_{t},w)_{\mathcal{T}_{h}}+\nu(\mathcal{K}_{h}(\xi_{h}^{u},\xi_{h}^{\widehat{u}}),\mathcal{K}_{h}(w,\mu))_{\mathcal{T}_{h}}+\mathcal{S}_{h}(\xi_{h}^{u},\xi_{h}^{\widehat{u}};w,\mu)\\
	+\mathcal{B}_{h}(\Pi_{k}^{o}u;\Pi_{k}^{o}u,\Pi_{l}^{\partial}u;w,\mu)-\mathcal{B}_{h}(u_{h};u_{h},\widehat{u}_{h};w,\mu)\\
	-((\Pi_{k}^{o}u-u)_{t},w)_{\mathcal{T}_{h}}-E_{L}(u,w,\mu)-E_{N}(u;u,w,\mu)
	&=0.   \label{4.5}
	\end{aligned}
	\right.
	\end{equation}
\end{myLem}

\begin{myLem} \label{L4.6}
Assume that $ u\in H^{2}(\Omega) $, then we have
\begin{equation} \label{4.6}
\mathcal{B}_{h}(\xi_{h}^{u};\Pi_{k}^{o}u,\Pi_{l}^{\partial}u;\xi_{h}^{u},\xi_{h}^{\widehat{u}})\lesssim \|\xi_ {h}^{u}\|_{0,\mathcal{T}_{h}}\cdotp\|u\|_{2}\cdotp|\|(\xi_{h}^{u},\xi_{h}^{\widehat{u}})\||, \forall (\xi_{h}^{u},\xi_{h}^{\widehat{u}})\in   V_{h}\times\widehat{V}_{h}.
\end{equation}
\end{myLem}
\begin{proof}
	From the definition of $ \mathcal{B}_{h} $, we have
	\begin{equation*}
	\begin{aligned}
	3\mathcal{B}_{h}(\xi_{h}^{u};\Pi_{k}^{o}u,\Pi_{l}^{\partial}u;\xi_{h}^{u},\xi_{h}^{\widehat{u}})&=-(\bm{b}(\xi_{h}^{u})\Pi_{k}^{o}u,\nabla \xi_{h}^{u})_{\mathcal{T}_{h}}+(\bm{b}(\xi_{h}^{u})\cdotp\nabla \Pi_{k}^{o}u,\xi_{h}^{u})_{\mathcal{T}_{h}}\\
	&\ \ \ -\left\langle \bm{b}(\xi_{h}^{u})\cdotp\bm{n}\Pi_{k}^{o}u,\xi_{h}^{\widehat{u}}\right\rangle_{\partial\mathcal{T}_{h}}+\left\langle \bm{b}(\xi_{h}^{u})\cdotp\bm{n}\Pi_{l}^{\partial}u,\xi_{h}^{u}\right\rangle_{\partial\mathcal{T}_{h}}\\
	&=-(\bm{b}(\xi_{h}^{u})\Pi_{k}^{o}u,\nabla \xi_{h}^{u})_{\mathcal{T}_{h}}+(\bm{b}(\xi_{h}^{u})\cdotp\nabla \Pi_{k}^{o}u,\xi_{h}^{u})_{\mathcal{T}_{h}}\\
	&\ \ \ +\left\langle \bm{b}(\xi_{h}^{u})\cdotp\bm{n}\Pi_{k}^{o}u,\xi_{h}^{u}-\xi_{h}^{\widehat{u}}\right\rangle_{\partial\mathcal{T}_{h}}-\left\langle \bm{b}(\xi_{h}^{u})\cdotp\bm{n}(\Pi_{k}^{o}u-\Pi_{l}^{\partial}u),\xi_{h}^{u}\right\rangle_{\partial\mathcal{T}_{h}}\\
	&=:\sum_{i=1}^{4}R_{i}.
	\end{aligned}
	\end{equation*}
Using the H\"{o}lder inequality, Lemmas \ref{L3.1},  \ref{L3.3} and   \ref{L3.4}, we can obtain
	\begin{equation*}
	\begin{aligned}
	R_{1}&=-(\bm{b}(\xi_{h}^{u})\Pi_{k}^{o}u,\nabla \xi_{h}^{u})_{\mathcal{T}_{h}} \lesssim \|\xi_{h}^{u}\|_{0,\mathcal{T}_{h}}|\Pi_{k}^{o}u|_{0,\infty,\mathcal{T}_{h}}\|\nabla\xi_{h}^{u}\|_{0,\mathcal{T}_{h}} \lesssim \|\xi_{h}^{u}\|_{0,\mathcal{T}_{h}}\|u\|_{2}|\|(\xi_{h}^{u},\xi_{h}^{\widehat{u}})\||,\\
	R_{2}&=(\bm{b}(\xi_{h}^{u})\cdotp\nabla \Pi_{k}^{o}u,\xi_{h}^{u})_{\mathcal{T}_{h}}\lesssim \|\xi_{h}^{u}\|_{0,\mathcal{T}_{h}}\|\nabla\Pi_{k}^{o}u\|_{0,6,\mathcal{T}_{h}}\|\xi_{h}^{u}\|_{0,3,\mathcal{T}_{h}} \lesssim \|\xi_{h}^{u}\|_{0,\mathcal{T}_{h}}\|u\|_{2}|\|(\xi_{h}^{u},\xi_{h}^{\widehat{u}})\||,\\
	R_{3}&=\left\langle \bm{b}(\xi_{h}^{u})\cdotp\bm{n}\Pi_{k}^{o}u,\xi_{h}^{u}-\xi_{h}^{\widehat{u}}\right\rangle_{\partial\mathcal{T}_{h}}
	\lesssim h^{-\frac{1}{2}}\|\xi_{h}^{u}\|_{0,\mathcal{T}_{h}}|\Pi_{k}^{o}u|_{0,\infty,\mathcal{T}_{h}}\|\xi_{h}^{u}-\xi_{h}^{\widehat{u}}\|_{0,2,\partial\mathcal{T}_{h}}\\
	&\lesssim \|\xi_{h}^{u}\|_{0,\mathcal{T}_{h}}\|u\|_{2}|\|(\xi_{h}^{u},\xi_{h}^{\widehat{u}})\||,\\
	R_{4}&=-\left\langle \bm{b}(\xi_{h}^{u})\cdotp\bm{n}(\Pi_{k}^{o}u-\Pi_{l}^{\partial}u),\xi_{h}^{u}\right\rangle_{\partial\mathcal{T}_{h}}\\
	&\leq h^{-\frac{1}{3}}\|\xi_{h}^{u}\|_{0,3,\mathcal{T}_{h}}\|\Pi_{k}^{o}u-\Pi_{l}^{\partial}u\|_{0,2,\partial\mathcal{T}_{h}}h^{-\frac{1}{6}}\|\xi_{h}^{u}\|_{0,6,\mathcal{T}_{h}}\\
	&\leq h^{-\frac{d}{6}}\|\xi_{h}^{u}\|_{0,\mathcal{T}_{h}}h^{-\frac{1}{2}}(\|\Pi_{k}^{o}u-u\|_{0,2,\partial\mathcal{T}_{h}}+\|u-\Pi_{l}^{\partial}u)\|_{0,2,\partial\mathcal{T}_{h}}|\|(\xi_{h}^{u},\xi_{h}^{\widehat{u}})\||\\
	&\lesssim \|\xi_{h}^{u}\|_{0,\mathcal{T}_{h}}\|u\|_{2}|\|(\xi_{h}^{u},\xi_{h}^{\widehat{u}})\||.
	\end{aligned}
	\end{equation*}
As a result, the estimate (\ref{4.6}) holds. 	
	
\end{proof}

\begin{myLem} \label{L4.7}
	Let $ (u,q) $ be the solution to the problem (\ref{2.5})  with $ u\in H^{1}(0,T;H^{k+1}(\Omega)) $ and $ u_{t}\in L^{2}(0,T;H^{k+1}(\Omega)) $, and let $ (\bm{q}_{h},u_{h},\widehat{u}_{h})\in  \bm{Q}_{h}\times V_{h}\times\widehat{V}_{h} $ be the solution of (\ref{2.12}). Then we have
	\begin{equation} \label{4.7}
	\|\xi_{h}^{u}(t)\|^{2}_{0,\mathcal{T}_{h}}+\nu\int_{0}^{t}|\|(\xi_{h}^{u}(\tau),\xi_{h}^{\widehat{u}}(\tau))\||^{2}d \tau\lesssim   h^{2k}(\|u(t)\|_{k+1}+\|u_t(t)\|_{k+1}),
	\end{equation}
%	where $ C $ is a positive constant only depends on   $ \nu $, $ H^{k+1} $ norm of $ u $ and $ u_{t} $ at each time.
\end{myLem}
\begin{proof}
	Taking $ (w,\mu)=(\xi_{h}^{u},\xi_{h}^{\widehat{u}}) $ in (\ref{4.5}) and using the antisymmetry of $ \mathcal{B}_{h} $, together with Cauchy-Schwarz inequality, Lemmas \ref{L4.2}, \ref{L4.4},   \ref{L4.6} and the  Young's inequality, we have
	\begin{equation*}\begin{aligned}
	&\ \ \ \frac{1}{2}\frac{d}{dt}\|\xi_{h}^{u}\|_{0,\mathcal{T}_{h}}^{2}+\nu|\|(\xi_{h}^{u},\xi_{h}^{\widehat{u}})\||^2\\
	&=((\Pi_{k}^{o}u-u)_{t},\xi_{h}^{u})_{\mathcal{T}_{h}}+E_{L}(u,\xi_{h}^{u},\xi_{h}^{\widehat{u}})+E_{N}(u;u,\xi_{h}^{u},\xi_{h}^{\widehat{u}})+\mathcal{B}_{h}(\xi_{h}^{u};\Pi_{k}^{o}u,\Pi_{l}^{\partial}u;\xi_{h}^{u},\xi_{h}^{\widehat{u}})\\
	&\lesssim h^{k+1}\|u_{t}\|_{k+1}\|\xi_{h}^{u}\|_{0,\mathcal{T}_{h}}+Ch^{k}(\|u\|_{k+1}\|u\|_{2}+\nu\|u\|_{k+1})|\|(\xi_{h}^{u},\xi_{h}^{\widehat{u}})\||\\
	&\ \ \ +C\|\xi_{h}^{u}\|_{0,\mathcal{T}_{h}}\|u\|_{2}|\|(\xi_{h}^{u},\xi_{h}^{\widehat{u}})\||\\
	&\leq \frac{Ch^{2k+2}}{2}\|u_{t}\|^{2}_{k+1}+\frac{1}{2}\|\xi_{h}^{u}\|^{2}_{0,\mathcal{T}_{h}}+\frac{Ch^{2k}}{\nu}(\|u\|_{k+1}\|u\|_{2}+\nu\|u\|_{k+1})^{2}+\frac{\nu}{4}|\|(\xi_{h}^{u},\xi_{h}^{\widehat{u}})\||^{2}\\
	&\ \ \ +\frac{C}{\nu}\|\xi_{h}^{u}\|^{2}_{0,\mathcal{T}_{h}}\|u\|^{2}_{2}+\frac{\nu}{4}|\|(\xi_{h}^{u},\xi_{h}^{\widehat{u}})\||^{2}.
	\end{aligned}\end{equation*}
	Here $C$ is a positive constant independent of $h$. 
	Integrating the above inequality with respect to $ t $ yields
	\begin{equation*}
	\|\xi_{h}^{u}(t)\|_{0,\mathcal{T}_{h}}^{2}+\nu\int_{0}^{t}|\|(\xi_{h}^{u},\xi_{h}^{\widehat{u}})\||^2d \tau\leq \|\xi_{h}^{u}(0)\|_{0,\mathcal{T}_{h}}^{2}+C\int_{0}^{t}\|\xi_{h}^{u}\|_{0,\mathcal{T}_{h}}^{2}d \tau+Ch^{2k},
	\end{equation*}
which, together with  the  Gronwall’s inequality,  gives the desired result (\ref{4.7}).
\end{proof}

From   Lemma \ref{L4.7} and the triangle inequality,  we easily  get the following error estimates for the semi-discrete scheme.
\begin{myTheo}
Let $ (\bm{q}(t),u(t)) $ and $ (\bm{q}_{h}(t),u_{h}(t)) $ be the solutions to the problem (\ref{2.5}) and (\ref{2.12}), respectively. Suppose that $ u\in H^{1}(0,T;H^{k+1}(\Omega)) $ and $ u_{t}\in L^{2}(0,T;H^{k+1}(\Omega)) $,  then 
it holds  
\begin{equation}
\|u(t)-u_{h}(t)\|_{0,\mathcal{T}_{h}}\lesssim h^{k}(\|u(t)\|_{k+1}+\|u_t(t)\|_{k+1}), \quad  t\in[0,T],
\end{equation}
  and
\begin{equation}
\left( \int_{0}^{T}\|\bm{q}-\bm{q}_{h}(\tau)\|_{0,\mathcal{T}_{h}}^{2}d \tau\right) ^{\frac{1}{2}}\lesssim h^{k} \left( \int_{0}^{T}(\|u(\tau)\|_{k+1}^2+\|u_t(\tau)\|_{k+1}^2)d \tau\right) ^{\frac{1}{2}}.
\end{equation}
%where $ C $ is a positive constant only depends on the $ \nu $, $ H^{k+1} $ norm of $ u $ and $ u_{t} $ at each time.
\end{myTheo}

\section{Fully discrete HDG   method}
\subsection{Backward Euler fully discrete scheme}

Given a positive integer $N$, let $0=t_0<t_1<...<t_{\mathrm{N}}=T$ be a uniform division of time domain $[0,T]$,  with   the time step $ \Delta t := \frac{T}{N} $. 
 We refer to $ \bm{q}_{h}^n,u_{h}^n,\widehat{u}_{h}^n $ as the approximation of $ \bm{q}_{h}(t_{n}),u_{h}(t_{n}),\widehat{u}_{h}(t_{n}) $ respectively at the discrete time $ t_{n} = n\Delta t $ for $ n = 1, 2, \ldots , N $. By replacing the time derivative $ (u_{h})_{t} $ at time $ t_{n} $ by the backward difference quotient $$ \partial_{t}u_{h}^{n}=\frac{u_{h}^{n}-u_{h}^{n-1}}{\Delta t} $$ in (\ref{2.12}), the linearized backward Euler HDG   scheme is given as follows: for
each $ 1 \leq n \leq N $, find $ (\bm{q}_{h}^n,u_{h}^n,\widehat{u}_{h}^n)\in \bm{Q}_{h}\times V_{h}\times\widehat{V}_{h} $ such that
\begin{equation}
\left\{
\begin{aligned}
\bm{q}^{n}_{h}+\mathcal{K}_{h}(u^{n}_{h},\widehat{u}^{n}_{h})&=0,\\
(\partial_{t}u^{n}_{h},w)_{\mathcal{T}_{h}}+\nu(\mathcal{K}_{h}(u^{n}_{h},\widehat{u}^{n}_{h}),\mathcal{K}_{h}(w,\mu))_{\mathcal{T}_{h}}+\mathcal{S}_{h}(u^{n}_{h},\widehat{u}^{n}_{h};w,\mu)\\
+\mathcal{B}_{h}(u^{n-1}_{h};u^{n}_{h},\widehat{u}^{n}_{h};w,\mu)&=(f^{n},w)_{\mathcal{T}_{h}},\\
u_{h}^{0}&=\Pi_{k}^{o}u_{0},  \label{5.1}
\end{aligned}
\right.
\end{equation}
%Taking $ v_{h}=u_{h} $, we can obtain
%\begin{equation}
%\|u_{h0}^{n}\|_{0}^{2}+\Delta t\nu|\|u_{h}^{n}\||^{2}\leq \|u_{h0}^{n-1}\|_{0}\|u_{h0}^{n}\|_{0}+\Delta t\|f\|_{0}\|u_{h0}\|_{0}
%\end{equation}
for all $ (w,\mu)\in  V_{h}\times\widehat{V}_{h} $.

%\subsection{Stability result}
\begin{myTheo}
	For the fully discrete scheme (\ref{5.1}), we have the following stability result: for any $ 1\leq n \leq N $,
	\begin{equation}
	\|u_{h}^{n}\|^{2}_{0,\mathcal{T}_{h}}+\sum_{i=1}^{n}\|u^{i}_{h}-u^{i-1}_{h}\|^{2}_{0,\mathcal{T}_{h}}+\nu\Delta t\sum_{i=1}^{n}|\|(u^{i}_{h},\widehat{u}^{i}_{h})\||^{2}\lesssim \|u_{h}^{0}\|^{2}_{0,\mathcal{T}_{h}}+\frac{\Delta t}{\nu}\sum_{i=1}^{n}\|f^{i}\|^{2}_{0,\mathcal{T}_{h}}. \label{5.2}
	\end{equation}	
\end{myTheo}
\begin{proof}
	Taking $ (w,\mu)=(u^{i}_{h},\widehat{u}^{i}_{h}) $ in (\ref{5.1}), we have
	\begin{equation*}
	\frac{1}{\Delta t}(u^{i}_{h}-u^{i-1}_{h},u^{i}_{h})_{\mathcal{T}_{h}}+\nu|\|(u^{i}_{h},\widehat{u}^{i}_{h})\||^{2}=(f^{i},u^{i}_{h})_{\mathcal{T}_{h}},
	\end{equation*}
	using the Cauchy-Schwarz inequality, the Young’s inequality and Lemma \ref{L3.4}, we get
	\begin{equation*}	
	\frac{1}{2\Delta t}(\|u^{i}_{h}\|^{2}_{0,\mathcal{T}_{h}}-\|u^{i-1}_{h}\|^{2}_{0,\mathcal{T}_{h}}+\|u^{i}_{h}-u^{i-1}_{h}\|^{2}_{0,\mathcal{T}_{h}})+\nu|\|(u^{i}_{h},\widehat{u}^{i}_{h})\||^{2}\leq \frac{C}{2\nu}\|f^{i}\|^{2}_{0,\mathcal{T}_{h}}+\frac{\nu}{2}|\|(u^{i}_{h},\widehat{u}^{i}_{h})\||^{2}.
	\end{equation*}
	Summing up the above inequality from $ i=1 $ to $ i=n $ leads to the desired result.
\end{proof}

%\subsection{Existence of a unique solution}
The following theorem shows a result  the  existence and uniqueness of the fully discrete solution. % the scheme  (\ref{5.1}).
\begin{myTheo}
	Given $ u_{h}^{n-1} $, the fully discrete scheme (\ref{5.1}) admits a unique solution $ (\bm{q}_{h}^n,u_{h}^n,\widehat{u}_{h}^n) $ for $1\leq n\leq N$.
\end{myTheo}
\begin{proof}
	As (\ref{5.1}) is a linear   square  system, we know that uniqueness is equivalent to existence. So we only need to show that 	if $ (\bm{q}_{1h}^n,u_{1h}^n,\widehat{u}_{1h}^n) $ and $ (\bm{q}_{2h}^n,u_{2h}^n,\widehat{u}_{2h}^n) $ are two solutions of (\ref{5.1}), then $ (\bm{q}_{1h}^n,u_{1h}^n,\widehat{u}_{1h}^n)=(\bm{q}_{2h}^n,u_{2h}^n,\widehat{u}_{2h}^n) $.
	
	In fact, we have
	\begin{equation}
	\begin{aligned}
	(\frac{u_{1h}^{n}-u_{h}^{n-1}}{\Delta t},w)_{\mathcal{T}_{h}}+\nu(\mathcal{K}_{h}(u_{1h}^n,\widehat{u}_{1h}^n),\mathcal{K}_{h}(w,\mu))_{\mathcal{T}_{h}}+\mathcal{S}_{h}(u_{1h}^n,\widehat{u}_{1h}^n;w,\mu)\\
	+\mathcal{B}_{h}(u^{n-1}_{h};u_{1h}^n,\widehat{u}_{1h}^n;w,\mu)&=(f^{n},w)_{\mathcal{T}_{h}},  \label{5.3}
	\end{aligned}
	\end{equation}
	\begin{equation}
	\begin{aligned}
	(\frac{u_{2h}^{n}-u_{h}^{n-1}}{\Delta t},w)_{\mathcal{T}_{h}}+\nu(\mathcal{K}_{h}(u_{2h}^n,\widehat{u}_{2h}^n),\mathcal{K}_{h}(w,\mu))_{\mathcal{T}_{h}}+\mathcal{S}_{h}(u_{2h}^n,\widehat{u}_{2h}^n;w,\mu)\\
	+\mathcal{B}_{h}(u^{n-1}_{h};u_{2h}^n,\widehat{u}_{2h}^n;w,\mu)&=(f^{n},w)_{\mathcal{T}_{h}}.  \label{5.4}
	\end{aligned}
	\end{equation}
	Subtracting (\ref{5.4}) from (\ref{5.3}), taking $ (w,\mu)=(u_{1h}^{n}-u_{2h}^{n},\widehat{u}_{1h}^{n}-\widehat{u}_{2h}^{n})=(\eta^{n},\widehat{\eta}^{n}) $ and using the antisymmetry of $ \mathcal{B}_{h} $, we get
	\begin{equation*}
	\frac{1}{\Delta t}(\eta^{n},\eta^{n})+\nu|\|(\eta^{n},\widehat{\eta}^{n})\||^{2}=0,
	\end{equation*}
	which means that $\eta^{n}=\widehat{\eta}^{n}=0$. This completes the proof. %Therefore, the solution $ (\bm{q}_{h}^n,u_{h}^n,\widehat{u}_{h}^n) $ to (\ref{5.1}) must also exist and be unique. 
\end{proof}

\subsection{A priori error estimation}
 We first recall the discrete version of the Gronwall inequality in a slightly more general form used in \cite{Heywood1990}.

\begin{myLem} \label{L5.2}
	Let $ \rho, G $ and $ a_{j}, b_{j}, c_{j}, \gamma_{j}, $ for integers $ j\geq 0 $, be nonnegative numbers such that
	\begin{equation*}
	a_{n}+\rho\sum_{j=0}^{n}b_{j}\leq \rho\sum_{j=0}^{n}\gamma_{j}a_{j}+\rho\sum_{j=0}^{n}c_{j}+G, \ \forall n\geq 0.
	\end{equation*}
    Suppose that $ \rho\gamma_{j}<1 $ for all $ j $, and set $ \sigma_{j}=(1-\rho\gamma_{j})^{-1} $, then
    \begin{equation*}
    a_{n}+\rho\sum_{j=0}^{n}b_{j}\leq \exp(\rho\sum_{j=0}^{n}\sigma_{j}\gamma_{j})(\rho\sum_{j=0}^{n}c_{j}+G), \ \forall n\geq 0.
    \end{equation*}
\end{myLem}

By following the same line as in the proof of Lemma \ref{L4.3}, we can derive the following lemma.
\begin{myLem} \label{L5.3}
	Let $ (u,q) $ be the solution to the problem (\ref{2.5}), then for any $ (w,\mu)\in V_{h}\times\widehat{V}_{h} $ it holds
	\begin{equation}
	\left\{
	\begin{aligned}
	\boldsymbol{\Pi}_{k-1}^{o}\bm{q}(t_{n})=&-\mathcal{K}_{h}(\Pi_{k}^{o}u(t_{n}),\Pi_{l}^{\partial}u(t_{n})),\\
	(\partial_{t}\Pi_{k}^{o}u(t_{n}),w)_{\mathcal{T}_{h}}&+\nu(\mathcal{K}_{h}(\Pi_{k}^{o}u(t_{n}),\Pi_{l}^{\partial}u(t_{n})),\mathcal{K}_{h}(w,\mu))_{\mathcal{T}_{h}}\\
	&+\mathcal{S}_{h}(\Pi_{k}^{o}u(t_{n}),\Pi_{l}^{\partial}u(t_{n});w,\mu)+\mathcal{B}_{h}(\Pi_{k}^{o}u(t_{n});\Pi_{k}^{o}u(t_{n}),\Pi_{l}^{\partial}u(t_{n});w,\mu)\\
	=&((\partial_{t}\Pi_{k}^{o}u(t_{n})-u_{t}(t_{n}),w)_{\mathcal{T}_{h}}+(f(t_{n}),w)\\
	&+E_{L}(u(t_{n}),w,\mu)+E_{N}(u(t_{n});u(t_{n}),w,\mu). \label{5.5}
	\end{aligned}
	\right.
	\end{equation}
\end{myLem}

\begin{myTheo} \label{T5.3}
Let $ (\bm{q}(t),u(t)) $ and $ (\bm{q}_{h}^n,u_{h}^n) $ be the solutions to the problem (\ref{2.5}) and (\ref{5.1}), respectively. Suppose $ u \in L^{\infty}(0, T; H^{k+1}(\Omega)) $, $ u_{t} \in L^{\infty}(0, T; H^{k+1}(\Omega)) $ and $ u_{tt} \in L^{2}(0, T; H^{k+1}(\Omega)) $. Then for any  $ 1\leq n \leq N $, it holds the   error estimate
\begin{eqnarray}
&&\|u(t_{n})-u_{h}^{n}\|_{0,\mathcal{T}_{h}}^{2}+\Delta t\sum_{j=1}^{n}\|\bm{q}(t_{n})-\bm{q}_{h}^n\|_{0,\mathcal{T}_{h}}^{2} \nonumber\\
&\lesssim& h^{2k}(\sum_{j=0}^{n}\|u(t_{j})\|_{k+1}^{2}+\int_{0}^{t_{n}}\|u_{t}\|^{2}_{k+1}ds)
+\Delta t^{2}(\int_{0}^{t_{n}}\|u_{tt}\|^{2}_{0}ds+\max\limits_{x\in\Omega,t\in[0,t_n]}|u_{t}|^{2}). \label{5.6}
\end{eqnarray}
%Here the hidden constant $ C $ in '$\lesssim$'   depends on $(\sum_{j=1}^{n}\|u(t_{j})\|_{k+1}^{2}$, $\int_{0}^{t_{n}}\|u_{t}\|^{2}_{k+1}ds)$ and $ \int_{0}^{t_{n}}\|u_{tt}\|^{2}_{0}ds$.
\end{myTheo}
\begin{proof}
Substracting (\ref{5.1}) from (\ref{5.5}),  for    any time $ t_{j} $ we have
\begin{equation}
\begin{aligned}
&(\partial_{t}\xi_{h}^{u}(t_{j}),w)_{\mathcal{T}_{h}}+\nu(\mathcal{K}_{h}(\xi_{h}^{u}(t_{j}),\xi_{h}^{\widehat{u}}(t_{j})),\mathcal{K}_{h}(w,\mu))_{\mathcal{T}_{h}}+\mathcal{S}_{h}(\xi_{h}^{u}(t_{j}),\xi_{h}^{\widehat{u}}(t_{j});w,\mu)\\
=&\mathcal{B}_{h}(u^{n-1}_{h};u^{n}_{h},\widehat{u}^{n}_{h};w,\mu)-\mathcal{B}_{h}(u^{n}_{h};u^{n}_{h},\widehat{u}^{n}_{h};w,\mu)\\
&+\mathcal{B}_{h}(u^{n}_{h};u^{n}_{h},\widehat{u}^{n}_{h};w,\mu)-\mathcal{B}_{h}(\Pi_{k}^{o}u(t_{n});\Pi_{k}^{o}u(t_{n}),\Pi_{l}^{\partial}u(t_{n});w,\mu)\\
&+(\partial_{t}\Pi_{k}^{o}u(t_{j})-u_{t}(t_{j}),w)_{\mathcal{T}_{h}}+E_{L}(u(t_{j}),w,\mu)+E_{N}(u(t_{j});u(t_{j}),w,\mu), \label{5.7}
\end{aligned}
\end{equation}
Taking $ (w,\mu)=(\xi_{h}^{u}(t_{j}),\xi_{h}^{\widehat{u}}(t_{j})) $ in (\ref{5.7}), we get
\begin{equation}
\begin{aligned}
&\frac{1}{\Delta t}(\xi_{h}^{u}(t_{j})-\xi_{h}^{u}(t_{j-1}),\xi_{h}^{u}(t_{j}))_{\mathcal{T}_{h}}+\nu|\|(\xi_{h}^{u}(t_{j}),\xi_{h}^{\widehat{u}}(t_{j}))\||^{2}\\
&=(\partial_{t}\Pi_{k}^{o}u(t_{j})-u_{t}(t_{j}),\xi_{h}^{u}(t_{j}))_{\mathcal{T}_{h}}+E_{L}(u(t_{j}),\xi_{h}^{u}(t_{j}),\xi_{h}^{\widehat{u}}(t_{j}))\\
&\ \ +E_{N}(u(t_{j});u(t_{j}),\xi_{h}^{u}(t_{j}),\xi_{h}^{\widehat{u}}(t_{j}))-\mathcal{B}_{h}(\xi_{h}^{u}(t_{j});\Pi_{k}^{o}u(t_{j}),\Pi_{l}^{\partial}u(t_{j});\xi_{h}^{u}(t_{j}),\xi_{h}^{\widehat{u}}(t_{j}))\\
&\ \ -\mathcal{B}_{h}(u^{j}_{h}-u^{j-1}_{h};\Pi_{k}^{o}u(t_{j}),\Pi_{l}^{\partial}u(t_{j});\xi_{h}^{u}(t_{j}),\xi_{h}^{\widehat{u}}(t_{j}))\\
&:=Q^{j}_{1}+Q^{j}_{2}+Q^{j}_{3}+Q^{j}_{4}+Q^{j}_{5}. \label{5.8}
\end{aligned}
\end{equation}
From the Cauchy-Schwarz inequality, it follows
\begin{equation*}
\begin{aligned}
Q_{1}^{j}&=(\partial_{t}\Pi_{k}^{o}u(t_{j})-u_{t}(t_{j}),\xi_{h}^{u}(t_{j}))_{\mathcal{T}_{h}}\\
&=(\partial_{t}\Pi_{k}^{o}u(t_{j})-\partial_{t}u(t_{j}),\xi_{h}^{u}(t_{j}))_{\mathcal{T}_{h}}+(\partial_{t}u(t_{j})-u_{t}(t_{j}),\xi_{h}^{u}(t_{j}))_{\mathcal{T}_{h}}\\
&\leq (\|\partial_{t}\Pi_{k}^{o}u(t_{j})-\partial_{t}u(t_{j})\|_{0,\mathcal{T}_{h}}+\|\partial_{t}u(t_{j})-u_{t}(t_{j})\|_{0,\mathcal{T}_{h}})\|\xi_{h}^{u}(t_{j})\|_{0,\mathcal{T}_{h}}.
\end{aligned}
\end{equation*}
Then using the property of projection yields
\begin{equation*}
\begin{aligned}
\|\partial_{t}\Pi_{k}^{o}u(t_{j})-\partial_{t}u(t_{j})\|_{0,\mathcal{T}_{h}}&=\|\Pi_{k}^{o}\partial_{t}u(t_{j})-\partial_{t}u(t_{j})\|_{0,\mathcal{T}_{h}}
=\frac{1}{\Delta t}\int_{t_{j-1}}^{t_{j}}|\Pi_{k}^{o}u_{t}-u_{t}|ds\\
&\lesssim \frac{h^{k+1}}{\Delta t}\int_{t_{j-1}}^{t_{j}}|u_{t}|_{k+1}ds\\
%&\leq \frac{h^{k+1}}{\Delta t}(\int_{t_{j-1}}^{t_{j}}|u_{t}|^{2}_{k+1}ds)^{\frac{1}{2}}(\int_{t_{j-1}}^{t_{j}}1^{2}ds)^{\frac{1}{2}}\\
&\lesssim\frac{h^{k+1}}{\sqrt{\Delta t}}(\int_{t_{j-1}}^{t_{j}}|u_{t}|^{2}_{k+1}ds)^{\frac{1}{2}}.
\end{aligned}
\end{equation*}
Similarly, we can get
\begin{equation*}
\begin{aligned}
\|\partial_{t}u(t_{j})-u_{t}(t_{j})\|_{0,\mathcal{T}_{h}}&=
\frac{1}{\Delta t}\int_{t_{j-1}}^{t_{j}}(s-t_{j-1})\|u_{tt}\|_{0}ds\\
&\leq\frac{1}{\Delta t}(\int_{t_{j-1}}^{t_{j}}(s-t_{j-1})^{2}ds)^{\frac{1}{2}}(\int_{t_{j-1}}^{t_{j}}\|u_{tt}\|^{2}_{0}ds)^{\frac{1}{2}}\\
&\leq \sqrt{\Delta t}(\int_{t_{j-1}}^{t_{j}}\|u_{tt}\|^{2}_{0}ds)^{\frac{1}{2}}.
\end{aligned}
\end{equation*}
So
\begin{equation*}
Q_{1}^{j}\lesssim (\frac{h^{k+1}}{\sqrt{\Delta t}}(\int_{t_{j-1}}^{t_{j}}|u_{t}|^{2}_{k+1}ds)^{\frac{1}{2}}+\sqrt{\Delta t}(\int_{t_{j-1}}^{t_{j}}\|u_{tt}\|^{2}_{0}ds)^{\frac{1}{2}})\|\xi_{h}^{u}(t_{j})\|_{0,\mathcal{T}_{h}}. \label{5.10}
\end{equation*}
By  Lemmas \ref{L4.2},   \ref{L4.4} and   \ref{L4.6}, we have
\begin{align*}
Q_{2}^{j}&\lesssim h^{k}\|u(t_{j})\|_{k+1} |\|(\xi_{h}^{u}(t_{j}),\xi_{h}^{\widehat{u}}(t_{j}))\||,\\
Q_{3}^{j}&\lesssim h^{k}\|u(t_{j})\|_{k+1} \|u(t_{j})\|_{2} |\|(\xi_{h}^{u}(t_{j}),\xi_{h}^{\widehat{u}}(t_{j}))\||,\\
Q^{j}_{4}&\lesssim \|\xi_{h}^{u}(t_{j})\|_{0,\mathcal{T}_{h}} \|u(t_{j})\|_{2} |\|(\xi_{h}^{u}(t_{j}),\xi_{h}^{\widehat{u}}(t_{j}))\||,\\
Q^{j}_{5}&\lesssim \|u^{j}_{h}-u^{j-1}_{h}\|_{0,\mathcal{T}_{h}} \|u(t_{j})\|_{2} |\|(\xi_{h}^{u}(t_{j}),\xi_{h}^{\widehat{u}}(t_{j}))\||. \label{5.11}
\end{align*}
From the triangle inequality and the property of   projection, it follows
\begin{equation*}
\begin{aligned}
\|u^{j}_{h}-u^{j-1}_{h}\|_{0,\mathcal{T}_{h}}&=\|u^{j}_{h}-\Pi_{k}^{o}u(t_{j})+\Pi_{k}^{o}u(t_{j-1})-u^{j-1}_{h}+\Pi_{k}^{o}(u(t_{j})-u(t_{j-1}))\|_{0,\mathcal{T}_{h}}\\
&\lesssim \|\xi_{h}^{u}(t_{j})-\xi_{h}^{u}(t_{j-1})\|_{0,\mathcal{T}_{h}}+\|\Pi_{k}^{o}(u(t_{j})-u(t_{j-1}))-(u(t_{j})-u(t_{j-1}))\|_{0,\mathcal{T}_{h}}\\
&\ \ +\|u(t_{j})-u(t_{j-1})\|_{0,\mathcal{T}_{h}}\\
&\lesssim \|\xi_{h}^{u}(t_{j})-\xi_{h}^{u}(t_{j-1})\|_{0,\mathcal{T}_{h}}+h^{k+1}\|u(t_{j})-u(t_{j-1})\|_{k+1}+\Delta t\max\limits_{x\in\Omega,t\in[t_{j-1},t_j]}|u_{t}|.
\end{aligned}
\end{equation*}
Thus, 
\begin{equation*}
\begin{aligned}
Q^{j}_{5}&\lesssim \left(\|\xi_{h}^{u}(t_{j})-\xi_{h}^{u}(t_{j-1})\|_{0,\mathcal{T}_{h}}+h^{k+1}\|u(t_{j})-u(t_{j-1})\|_{k+1}\right.\\
&\quad \left.+\Delta t\max\limits_{x\in\Omega, t\in[t_{j-1},t_j]}|u_{t}|\right)\cdotp\|u(t_{j})\|_{2}\cdotp|\|(\xi_{h}^{u}(t_{j}),\xi_{h}^{\widehat{u}}(t_{j}))\||.  \label{5.15}
\end{aligned}
\end{equation*}
Substituting the estimates of $Q^{j}_{m}$ ($m=1,2,\cdots,5$) into  \eqref{5.8},   summing up the obtained inequality from $ j=1 $ to $ j=n $, and noticing that
\begin{equation*}
\frac{1}{\Delta t}(\xi_{h}^{u}(t_{j})-\xi_{h}^{u}(t_{j-1}),\xi_{h}^{u}(t_{j}))_{\mathcal{T}_{h}}=\frac{1}{2\Delta t}(\|\xi_{h}^{u}(t_{j})\|_{0,\mathcal{T}_{h}}^{2}-\|\xi_{h}^{u}(t_{j-1})\|_{0,\mathcal{T}_{h}}^{2}+\|\xi_{h}^{u}(t_{j})-\xi_{h}^{u}(t_{j-1})\|_{0,\mathcal{T}_{h}}^{2}), \label{5.9}
\end{equation*}
we get
\begin{equation}
\begin{aligned}
&\ \ \ \|\xi_{h}^{u}(t_{n})\|_{0,\mathcal{T}_{h}}^{2}+\sum_{j=1}^{n}\|\xi_{h}^{u}(t_{j})-\xi_{h}^{u}(t_{j-1})\|_{0,\mathcal{T}_{h}}^{2}+2\nu\Delta t\sum_{j=1}^{n}|\|(\xi_{h}^{u}(t_{j}),\xi_{h}^{\widehat{u}}(t_{j}))\||^{2}\\
&\lesssim \Delta t\sum_{j=1}^{n}(\frac{h^{k+1}}{\sqrt{\Delta t}}(\int_{t_{j-1}}^{t_{j}}|u_{t}|^{2}_{k+1}ds)^{\frac{1}{2}}+\sqrt{\Delta t}(\int_{t_{j-1}}^{t_{j}}\|u_{tt}\|^{2}_{0}ds)^{\frac{1}{2}})\|\xi_{h}^{u}(t_{j})\|_{0,\mathcal{T}_{h}}\\
&\ \ \ +\Delta t\sum_{j=1}^{n}(h^{k}(\|u(t_{j})\|_{k+1}+\|u(t_{j-1})\|_{k+1})+\Delta t\max\limits_{x\in\Omega,t\in[0,t_n]}|u_{t}|+\|\xi_{h}^{u}(t_{j})\|_{0,\mathcal{T}_{h}}\\
&\ \ \ +\|\xi_{h}^{u}(t_{j})-\xi_{h}^{u}(t_{j-1})\|_{0,\mathcal{T}_{h}})|\|(\xi_{h}^{u}(t_{j}),\xi_{h}^{\widehat{u}}(t_{j}))\||\\
&\lesssim h^{2k+2}\int_{0}^{t_{n}}|u_{t}|^{2}_{k+1}ds+\Delta t^{2}\int_{0}^{t_{n}}\|u_{tt}\|^{2}_{0}ds+\Delta t\sum_{j=1}^{n}\|\xi_{h}^{u}(t_{j})\|_{0,\mathcal{T}_{h}}^{2}\\
&\ \ \ +h^{2k}\sum_{j=0}^{n}\|u(t_{j})\|_{k+1}^{2}+\Delta t^2\max\limits_{x\in\Omega,t\in[0,t_n]}|u_{t}|^{2}+\Delta t\sum_{j=1}^{n}\|\xi_{h}^{u}(t_{j})\|_{0,\mathcal{T}_{h}}^{2}+\Delta t\sum_{j=1}^{n}\|\xi_{h}^{u}(t_{j})-\xi_{h}^{u}(t_{j-1})\|_{0,\mathcal{T}_{h}}^{2}\\
&\ \ \ +\nu\Delta t\sum_{j=1}^{n}|\|(\xi_{h}^{u}(t_{j}),\xi_{h}^{\widehat{u}}(t_{j}))\||^{2},
\end{aligned}
\end{equation}
which, together with Lemma \ref{L5.2} and the triangle inequality,  indicates the desired result.
\end{proof}

%\subsection{Diagonally implicit Runge–Kutta fully discrete HDG schemes}
\begin{rem}
	Due to the use of   backward Euler scheme for the temporal discretization,  the  fully discretization (\ref{5.1}) is only of first order temporal  accuracy.  In fact,  we can also  apply other higher order implicit time-stepping
	schemes such as the diagonally implicit Runge–Kutta (DIRK) methods.
	
	Consider the following two-stage and third-order DIRK(2,3) formulas \cite{Alexande1977} written in the form of Butcher’s table for time integration:
	\begin{center}
	\begin{tabular}{c c|c}
		$ a_{11} $ & $ a_{12} $ & $ c_{1} $ \\
		
		$ a_{21} $ & $ a_{22} $ & $ c_{2} $ \\
		\hline
		$ b_{1} $ & $ b_{2} $ &  \\	
	\end{tabular}
	\end{center}
	%\begin{rem}
	%The DIRK(2,3) scheme proposed by Crouzeix [?] is known to be A-stable.
	%\end{rem}
We   apply the DIRK(2,3) method to the semi-discrete
	HDG scheme (\ref{2.12}). To simplify notation we write 
	\begin{itemize}
	\item $ t^{n,i} $ for $ t^{n}+c_{i}\Delta t $, 
	\item $ \bm{y}_{h}^{n} $ for $ (\bm{q}^{n}_{h},u^{n}_{h}) $, and 
	\item $ \bm{y}_{h}^{n,i} $ for $ (\bm{q}^{n,i}_{h},u^{n,i}_{h})=(\bm{q}_{h}(t^{n,i}),u_{h}(t^{n,i})) $. 
	\end{itemize}
	The numerical solution $ \bm{y}_{h}^{n+1}=(\bm{q}^{n+1}_{h},u^{n+1}_{h}) $ at time level $ n + 1 $ given by the DIRK(2,3) method is
	computed as follows:
	\begin{equation}
	\bm{y}_{h}^{n+1}=\bm{y}_{h}^{n}+\Delta t(b_{1}\bm{f}_{h,1}+b_{2}\bm{f}_{h,2}),  \label{5.16}
	\end{equation}
	where
	\begin{equation}
	\begin{aligned}
	\bm{f}_{h,1} = \frac{\bm{y}_{h}^{n,1}-\bm{y}_{h}^{n}}{a_{11}\Delta t},\quad 
	\bm{f}_{h,2} = \frac{\bm{y}_{h}^{n,2}-\bm{y}_{h}^{n}}{a_{22}\Delta t}-\frac{a_{21}}{a_{22}}\bm{f}_{h,1}.  \label{5.17}
	\end{aligned}
	\end{equation}
	The intermediate states $ \bm{y}_{h}^{n,i}=(\bm{q}^{n,i}_{h},u^{n,i}_{h}) $, $ i = 1,2 $, with the Ossen iteration, are determined as follows: given $ u_{h}^{n,i,0} $, find $ (\bm{q}_{h}^{n,i},u_{h}^{n,i},\widehat{u}_{h}^{n,i})=(\bm{q}_{h}^{n,i,r},u_{h}^{n,i,r},\widehat{u}_{h}^{n,i,r})\in \bm{Q}_{h}\times V_{h}\times\widehat{V}_{h} $ such that for $ r=1,2,\ldots $,
	\begin{equation}
	\left\{
	\begin{aligned}
	\bm{q}^{n,i,r}_{h}+\mathcal{K}_{h}(u^{n,i,r}_{h},\widehat{u}^{n,i,r}_{h})&=0,\\
	(\frac{1}{a_{ii}\Delta t}u^{n,i,r}_{h},w)_{\mathcal{T}_{h}}+\nu(\mathcal{K}_{h}(u^{n,i,r}_{h},\widehat{u}^{n,i,r}_{h}),\mathcal{K}_{h}(w,\mu))_{\mathcal{T}_{h}}\\+\mathcal{S}_{h}(u^{n,i,r}_{h},\widehat{u}^{n,i,r}_{h};w,\mu)
	+\mathcal{B}_{h}(u^{n,i,r-1}_{h};u^{n,i,r}_{h},\widehat{u}^{n,i,r}_{h};w,\mu)&=(f(t^{n,i}),w)_{\mathcal{T}_{h}}+(z_{h}^{n,i},w)_{\mathcal{T}_{h}},\\
	u_{h}^{0}&=\Pi_{k}^{o}u_{0},   \label{5.18}
	\end{aligned}
	\right.
	\end{equation}
	and the terms $ z_{h}^{n,i} $, $ i = 1,2 $, on the right-hand side of (\ref{5.18}) are given by
	\begin{equation*}
	\begin{aligned}
	z_{h}^{n,1} = \frac{u_{h}^{n}}{a_{11}\Delta t},\quad 
	z_{h}^{n,2} = \frac{u_{h}^{n}}{a_{22}\Delta t}+\frac{a_{21}}{a_{22}}(\frac{u_{h}^{n,1}}{a_{11}\Delta t}-z_{h}^{n,1}).
	\end{aligned}
	\end{equation*}
%	We note that the resulting system (\ref{5.4}) at each ith stage of the DIRK(2,3) method is very similar to the system (\ref{5.1}) of the Backward–Euler method. Therefore, most of our discussion for the Backward–Euler method would be relevant to the DIRK(2,3)
%	method.
We note that the resulting system (\ref{5.18}) at each $(i)$-th stage of the DIRK(2,3) method is very similar to the backward–Euler  system (\ref{5.1}). And we will give some numerical experiments in next section to show the efficiency of DIRK(2,3) fully discrete scheme  (\ref{5.16}).

\end{rem}

\section{Numerical experiments}
In this section, we present some numerical results to demonstrate the performance of our proposed  fully discrete HDG
schemes for solving   the Burgers' equation. 
%In all the numerical experiments, we use Oseen's iteration scheme (5.1) with initial
%guess $ u_{h}^{n,0} = 0 $ to linearize
%the nonlinear discrete problem.

We consider two cases of the HDG methods with $ k \geq 1 $:
\begin{equation}
\begin{aligned}
HDG-I&:\ l=k,\\
HDG-II&:\ l=k-1.
\end{aligned}
\end{equation}
\begin{exmp} \label{E6.1} This example  is to test the accuracy of the back Euler fully discrete scheme (\ref{5.1}).
Take $ \Omega=[0,1]\times[0,1] $, $ T=1 $, and  $ \nu =1, 0.01 $.  The exact solution to the problem (\ref{1.1a})-(\ref{1.1c}) is given by
	\begin{equation*}
	u=e^{-t}x(x-1)y(y-1),\ \ \ \ in\ \Omega\times[0,T],
	\end{equation*}
%the parameter $ \nu $ is taken as $ 1, 0.01 $, then 
Then the force term and the
boundary condition can be derived explicitly. 

We use $ M\times M $ uniform triangular meshes (c.f. Figure \ref{domain1}) for the spatial discretization and, to verify the spatial accuracy,   take the time step as %For  the scheme (\ref{5.1}), we 
% the time  and space steps satisfying 
 $ \Delta t=h^{2} /2$ (i.e. $N=M^2$) for $ k=1 $ and $ \Delta t=\sqrt{2}h^{3}/4 $ for $ k=2 $, respectively.  %The  errors  at the final time $ T $ are given. 

Some numerical results of   the relative errors for  the   approximations of $u$ and $\bm{q}$ at the final time with $ k = 1, 2 $ are shown in Tables \ref{table1}- \ref{table4}. We can see that the scheme (\ref{5.1})
yields $(k+1)$-th and $(k)$-th spatial convergence orders of $\lVert u(T)-u_{h}(T)\rVert_0$ and $\lVert \bm{q}(T)-\bm{q}_{h}(T)\rVert_0$, respectively. The results of $\lVert u(T)-u_{h}(T)\rVert_0$ are conformable to Theorems \ref{T5.3}.

\end{exmp}
\begin{figure}[htp]
	\centering
	\includegraphics[width=6.5cm,height=5cm]{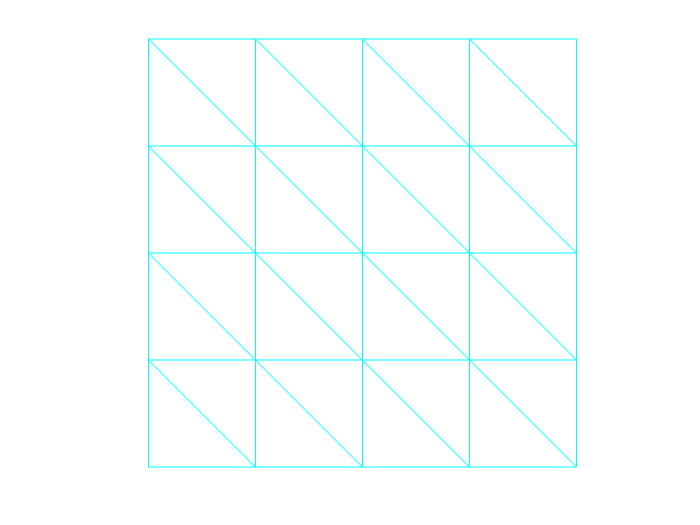}
	\includegraphics[width=6.5cm,height=5cm]{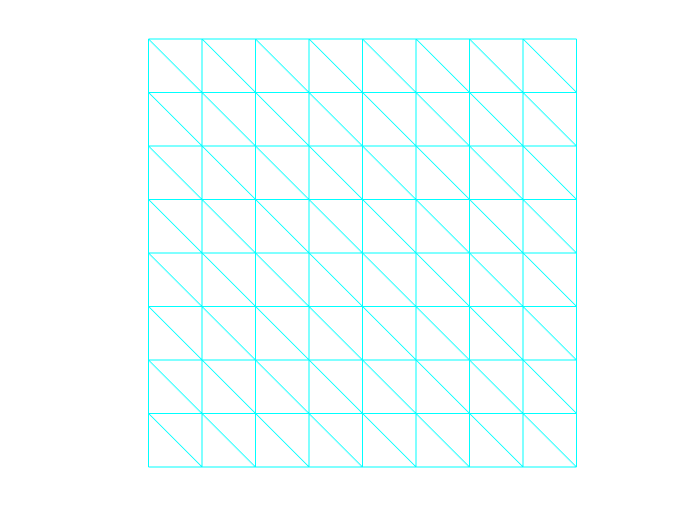} 
	\caption{The domain : $ 4\times 4 $(left) and $ 8\times 8 $(right) mesh} \label{domain1}
\end{figure}

\begin{table}[H]
	\normalsize
	\caption{History of convergence for Example \ref{E6.1} with  $\nu = 1, k = 1$  }
	\centering
	\label{table1}
	\footnotesize
	\subtable[Method: HDG-I$ (l=1) $]{
		\begin{tabular}{p{1.5cm}<{\centering}|p{1.6cm}<{\centering}|p{0.8cm}<{\centering}|p{1.6cm}<{\centering}|p{0.8cm}<{\centering}}
			\hline   
			\multirow{2}{*}{mesh}&  
			\multicolumn{2}{c|}{$\frac{\lVert u(T)-u_{h}(T) \rVert_0}{\lVert  u(T)\rVert_0}$ }&\multicolumn{2}{c}{$\frac{\lVert \bm{q}(T)- \bm{q}_{h}(T)\rVert_0}{\lVert \bm{q}(T)\rVert_0}$}\cr\cline{2-5}  			
			&error&order&error&order\cr  
			\cline{1-5}
			$ 4\times 4 $   &2.1597e-01      &--       &2.9311e-01     &--    \\
			\hline
			$ 8\times 8 $	&5.4132e-02      &2.00	    &1.4864e-01	    &0.98    \\
			\hline
			$ 16\times 16 $	 &1.3543e-02     &2.00        &7.4578e-02       &0.99     \\
			\hline
			$32\times 32$	&3.3865e-03       &2.00	     &3.7322e-02      &1.00	 \\
			\hline
			$64\times 64$	&8.4665e-04       &2.00	    &1.8665e-02	     &1.00	  \\
			\hline
		\end{tabular}
	}
	
	\subtable[Method: HDG-II$ (l=0) $]{
		\begin{tabular}{p{1.5cm}<{\centering}|p{1.6cm}<{\centering}|p{0.8cm}<{\centering}|p{1.6cm}<{\centering}|p{0.8cm}<{\centering}}
			\hline   
			\multirow{2}{*}{mesh}&  
			\multicolumn{2}{c|}{$\frac{\lVert u(T)-u_{h}(T) \rVert_0}{\lVert  u(T)\rVert_0}$ }&\multicolumn{2}{c}{$\frac{\lVert \bm{q}(T)- \bm{q}_{h}(T)\rVert_0}{\lVert \bm{q}(T)\rVert_0}$}\cr\cline{2-5}  
			&error&order&error&order\cr  
			\cline{1-5}
			$ 4\times 4 $   &2.4145e-01      &--       &3.1279e-01     &--    \\
			\hline
			$ 8\times 8 $	&6.0180e-02      &2.00	    &1.5806e-01	    &0.98    \\
			\hline
			$ 16\times 16 $	 &1.5038e-02     &2.00        &7.9255e-02       &1.00     \\
			\hline
			$32\times 32$	&3.7593e-03       &2.00	     &3.9656e-02      &1.00	 \\
			\hline
			$64\times 64$	&9.3980e-04       &2.00	    &1.9832e-02	     &1.00	  \\
			\hline
		\end{tabular}
	}	
\end{table}

\begin{table}[H]
	\normalsize
	\caption{History of convergence for Example \ref{E6.1} with  $\nu = 0.01, k = 1$ }
	\centering
	\label{table2}
	\footnotesize
	\subtable[Method: HDG-I$ (l=1) $]{
		\begin{tabular}{p{1.5cm}<{\centering}|p{1.6cm}<{\centering}|p{0.8cm}<{\centering}|p{1.6cm}<{\centering}|p{0.8cm}<{\centering}}
			\hline   
			\multirow{2}{*}{mesh}&  
			\multicolumn{2}{c|}{$\frac{\lVert u(T)-u_{h}(T) \rVert_0}{\lVert  u(T)\rVert_0}$ }&\multicolumn{2}{c}{$\frac{\lVert \bm{q}(T)- \bm{q}_{h}(T)\rVert_0}{\lVert \bm{q}(T)\rVert_0}$}\cr\cline{2-5}  
			&error&order&error&order\cr  
			\cline{1-5}
			$ 4\times 4 $   &1.1207e-01      &--       &2.9063e-01     &--    \\
			\hline
			$ 8\times 8 $	&3.3444e-02      &1.74	    &1.4978e-01	    &0.96    \\
			\hline
			$ 16\times 16 $	 &8.5650e-03     &1.97        &7.4851e-02       &1.00     \\
			\hline
			$32\times 32$	&2.1460e-03       &2.00	     &3.7359e-02      &1.00	 \\
			\hline
			$64\times 64$	&5.3674e-04       &2.00	    &1.8669e-02	     &1.00	  \\
			\hline
		\end{tabular}
	}
	
	\subtable[Method: HDG-II$ (l=0) $]{
		\begin{tabular}{p{1.5cm}<{\centering}|p{1.6cm}<{\centering}|p{0.8cm}<{\centering}|p{1.6cm}<{\centering}|p{0.8cm}<{\centering}}
			\hline   
			\multirow{2}{*}{mesh}&  
			\multicolumn{2}{c|}{$\frac{\lVert u(T)-u_{h}(T) \rVert_0}{\lVert  u(T)\rVert_0}$ }&\multicolumn{2}{c}{$\frac{\lVert \bm{q}(T)- \bm{q}_{h}(T)\rVert_0}{\lVert \bm{q}(T)\rVert_0}$}\cr\cline{2-5}    
			&error&order&error&order\cr  
			\cline{1-5}
			$ 4\times 4 $   &1.3157e-01      &--       &3.0275e-01     &--    \\
			\hline
			$ 8\times 8 $	&4.1890e-02      &1.65	    &1.6350e-01	    &0.89    \\
			\hline
			$ 16\times 16 $	 &1.0394e-02     &2.01        &8.0699e-02       &1.02     \\
			\hline
			$32\times 32$	&2.5616e-03       &2.02	     &3.9968e-02      &1.01	 \\
			\hline
			$64\times 64$	&6.3806e-04       &2.01	    &1.9936e-02	     &1.00	  \\
			\hline
		\end{tabular}
	}	
\end{table}

\begin{table}[H]
	\normalsize
	\caption{History of convergence for Example \ref{E6.1} with  $\nu = 1, k = 2$ }
	\centering
	\label{table3}
	\footnotesize
	\subtable[Method: HDG-I$ (l=2) $]{
		\begin{tabular}{p{1.5cm}<{\centering}|p{1.6cm}<{\centering}|p{0.8cm}<{\centering}|p{1.6cm}<{\centering}|p{0.8cm}<{\centering}}
			\hline   
			\multirow{2}{*}{mesh}&  
			\multicolumn{2}{c|}{$\frac{\lVert u(T)-u_{h}(T) \rVert_0}{\lVert  u(T)\rVert_0}$ }&\multicolumn{2}{c}{$\frac{\lVert \bm{q}(T)- \bm{q}_{h}(T)\rVert_0}{\lVert \bm{q}(T)\rVert_0}$}\cr\cline{2-5}  
			&error&order&error&order\cr  
			\cline{1-5}
			$ 4\times 4 $   &1.3700e-02      &--       &4.1763e-02     &--    \\
			\hline
			$ 8\times 8 $	&1.5937e-03      &3.10	    &1.0597e-02	    &1.98    \\
			\hline
			$ 16\times 16 $	 &1.9284e-04     &3.05        &2.6642e-03       &1.99     \\
			\hline
			$32\times 32$	&2.3736e-05       &3.02	     &6.6756e-04      &2.00	 \\
			\hline
			$64\times 64$	&2.9449e-06       &3.01	    &1.6705e-04	     &2.00	  \\
			\hline
		\end{tabular}
	}
	
	\subtable[Method: HDG-II$ (l=1) $]{
		\begin{tabular}{p{1.5cm}<{\centering}|p{1.6cm}<{\centering}|p{0.8cm}<{\centering}|p{1.6cm}<{\centering}|p{0.8cm}<{\centering}}
			\hline   
			\multirow{2}{*}{mesh}&  
			\multicolumn{2}{c|}{$\frac{\lVert u(T)-u_{h}(T) \rVert_0}{\lVert  u(T)\rVert_0}$ }&\multicolumn{2}{c}{$\frac{\lVert \bm{q}(T)- \bm{q}_{h}(T)\rVert_0}{\lVert \bm{q}(T)\rVert_0}$}\cr\cline{2-5}    
			&error&order&error&order\cr  
			\cline{1-5}
			$ 4\times 4 $   &1.4714e-02      &--       &4.3007e-02     &--    \\
			\hline
			$ 8\times 8 $	&1.7267e-03      &3.09	    &1.0913e-02	    &1.98    \\
			\hline
			$ 16\times 16 $	 &2.1016e-04     &3.04        &2.7449e-03       &1.99     \\
			\hline
			$32\times 32$	&2.5953e-05       &3.02	     &6.8800e-04      &2.00	 \\
			\hline
			$64\times 64$	&3.2256e-06       &3.01	    &1.7220e-04	     &2.00	  \\
			\hline
		\end{tabular}
	}	
\end{table}

\begin{table}[H]
	\normalsize
	\caption{History of convergence for Example \ref{E6.1} with  $\nu = 0.01, k = 2$ }
	\label{table4}
	\centering
	\footnotesize
	\subtable[Method: HDG-I$ (l=2) $]{
		\begin{tabular}{p{1.5cm}<{\centering}|p{1.6cm}<{\centering}|p{0.8cm}<{\centering}|p{1.6cm}<{\centering}|p{0.8cm}<{\centering}}
			\hline   
			\multirow{2}{*}{mesh}&  
			\multicolumn{2}{c|}{$\frac{\lVert u(T)-u_{h}(T) \rVert_0}{\lVert  u(T)\rVert_0}$ }&\multicolumn{2}{c}{$\frac{\lVert \bm{q}(T)- \bm{q}_{h}(T)\rVert_0}{\lVert \bm{q}(T)\rVert_0}$}\cr\cline{2-5}  
			&error&order&error&order\cr  
			\cline{1-5}
			$ 4\times 4 $   &9.4294e-02      &--       &1.0218e-01     &--    \\
			\hline
			$ 8\times 8 $	&1.2210e-02      &2.95	    &1.6179e-02	    &2.66    \\
			\hline
			$ 16\times 16 $	 &1.5272e-03     &3.00        &3.0797e-03       &2.39     \\
			\hline
			$32\times 32$	&1.9068e-04       &3.00	     &6.9507e-04      &2.15	 \\
			\hline
			$64\times 64$	&2.3817e-05       &3.00	    &1.6880e-04	     &2.04	  \\
			\hline
		\end{tabular}
	}
	
	\subtable[Method: HDG-II$ (l=1) $]{
		\begin{tabular}{p{1.5cm}<{\centering}|p{1.6cm}<{\centering}|p{0.8cm}<{\centering}|p{1.6cm}<{\centering}|p{0.8cm}<{\centering}}
			\hline   
			\multirow{2}{*}{mesh}&  
			\multicolumn{2}{c|}{$\frac{\lVert u(T)-u_{h}(T) \rVert_0}{\lVert  u(T)\rVert_0}$ }&\multicolumn{2}{c}{$\frac{\lVert \bm{q}(T)- \bm{q}_{h}(T)\rVert_0}{\lVert \bm{q}(T)\rVert_0}$}\cr\cline{2-5}    
			&error&order&error&order\cr  
			\cline{1-5}
			$ 4\times 4 $   &9.4922e-02      &--       &1.0298e-01     &--    \\
			\hline
			$ 8\times 8 $	&1.2259e-02      &2.95	    &1.6577e-02	    &2.64    \\
			\hline
			$ 16\times 16 $	 &1.5310e-03     &3.00        &3.1770e-03       &2.38     \\
			\hline
			$32\times 32$	&1.9114e-04       &3.00	     &7.2041e-04      &2.14	 \\
			\hline
			$64\times 64$	&2.3877e-05       &3.00	    &1.7525e-04	     &2.04	  \\
			\hline
		\end{tabular}
	}	
\end{table}

\begin{exmp} \label{E6.2} This example is to test the accuracy of the DIRK(2,3)   scheme  (\ref{5.16}).
	Take $ \Omega=[0,1]\times[0,1] $, $ T=1 $ and $ \nu=0.1 $. The exact solution to the problem (\ref{1.1a})-(\ref{1.1c}) is given by
	\begin{equation*}
	u=(e^{t}-1)xy\tanh(\frac{1-x}{\nu})\tanh(\frac{1-y}{\nu}),\ \ \ \ in\ \Omega\times[0,T].
	\end{equation*}
	%The source term $ f $ is then determined from this exact solution with .
	  We use $ M\times M $ uniform triangular spatial meshes(c.f. Figure \ref{domain1}) for the computation.  
	 
	 To test the temporal accuracy, we take $k=3$ and use a very fine spatial mesh with $N={256} $.  Numerical results of the errors at the final time $ T $   Table \ref{T5} show that the temporal convergence rate of  the scheme is close to third order. 
	
To verify the spatial accuracy,  we  
	adopt a  small time step $ \Delta t=0.005 $ so that the overall error is governed by the spatial error. Numerical results in Tables \ref{T6} and \ref{T7} for $ k=1 $ and $ k=2 $ show that scheme gives $(k+1)$-th and $(k)$-th spatial convergence orders of $\lVert u(T)-u_{h}(T)\rVert_0$ and $\lVert \bm{q}(T)-\bm{q}_{h}(T)\rVert_0$, respectively.
\end{exmp}

\begin{table}[H]
	\normalsize
	\caption{History of convergence with $ k=3 $: Example \ref{E6.2} }
	\label{T5}
	\centering
	\footnotesize
	
	\begin{tabular}{p{1.5cm}<{\centering}|p{1.6cm}<{\centering}|p{0.8cm}<{\centering}|p{1.6cm}<{\centering}|p{0.8cm}<{\centering}}
		\hline   
		\multirow{3}{*}{$ \Delta t $}&  
		\multicolumn{2}{c|}{HDG-I$ (l=3) $ }&\multicolumn{2}{c}{HDG-II$ (l=2) $}\cr\cline{2-5}
		&\multicolumn{2}{c|}{$\frac{\lVert \bm{u}(T)-\bm{u}_{h}(T)\rVert_0}{\lVert \bm{u}(T)\rVert_0}$ }&\multicolumn{2}{c}{$\frac{\lVert \bm{u}(T)-\bm{u}_{h}(T)\rVert_0}{\lVert \bm{u}(T)\rVert_0}$}\cr\cline{2-5}  			
		&error&order&error&order\cr  
		\cline{1-5}
		0.2	&2.2145e-03     &--    	&2.2145e-03	    &--	  \\
		\hline
		0.1   &3.7353e-04     &2.57    	&3.7353e-04	    &2.57	  \\
		\hline
		0.05	&5.7074e-05     &2.71    	&5.7074e-05	    &2.71	  \\  
		\hline
		0.025	 &8.1013e-06     &2.82    	&8.1013e-06	    &2.82	  \\
		\hline
		0.0125	 &1.0947e-06     &2.89    	&1.0947e-06	    &2.89	  \\
		\hline
		
	\end{tabular}
	
\end{table}

\begin{table}[H]
	\normalsize
	\caption{History of convergence with $ k=1 $: Example \ref{E6.2} }
	\label{T6}
	\centering
	\footnotesize
	\subtable[Method: HDG-I$ (l=1) $]{
		\begin{tabular}{p{1.5cm}<{\centering}|p{1.6cm}<{\centering}|p{0.8cm}<{\centering}|p{1.6cm}<{\centering}|p{0.8cm}<{\centering}}
			\hline   
			\multirow{2}{*}{mesh}&  
			\multicolumn{2}{c|}{$\frac{\lVert u(T)-u_{h}(T) \rVert_0}{\lVert  u(T)\rVert_0}$ }&\multicolumn{2}{c}{$\frac{\lVert \bm{q}(T)- \bm{q}_{h}(T)\rVert_0}{\lVert \bm{q}(T)\rVert_0}$}\cr\cline{2-5}    			
			&error&order&error&order\cr  
			\cline{1-5}
			$ 8\times 8 $	&1.5445e-01      &--	    &4.2222e-01	    &--    \\
			\hline
			$ 16\times 16 $	 &4.2801e-02     &1.85        &1.9709e-01       &1.10     \\
			\hline
			$32\times 32$	&1.0908e-02       &1.97	     &9.8335e-02      &1.00	 \\
			\hline
			$64\times 64$	&2.7412e-03       &1.99	    &4.9190e-02	     &1.00	  \\
			\hline
			$128\times 128$	&6.8606e-04       &2.00	    &2.4598e-02	     &1.00	  \\
			\hline
		\end{tabular}
	}
	
	\subtable[Method: HDG-II$ (l=0) $]{
		\begin{tabular}{p{1.5cm}<{\centering}|p{1.6cm}<{\centering}|p{0.8cm}<{\centering}|p{1.6cm}<{\centering}|p{0.8cm}<{\centering}}
			\hline   
			\multirow{2}{*}{mesh}&  
			\multicolumn{2}{c|}{$\frac{\lVert u(T)-u_{h}(T) \rVert_0}{\lVert  u(T)\rVert_0}$ }&\multicolumn{2}{c}{$\frac{\lVert \bm{q}(T)- \bm{q}_{h}(T)\rVert_0}{\lVert \bm{q}(T)\rVert_0}$}\cr\cline{2-5}  
			&error&order&error&order\cr  
			\cline{1-5}
			$ 8\times 8 $	&1.7900e-01      &--	    &4.5177e-01	    &--    \\
			\hline
			$ 16\times 16 $	 &4.7423e-02     &1.92        &2.0942e-01       &1.11     \\
			\hline
			$32\times 32$	&1.1971e-02       &1.99	     &1.0380e-01      &1.01	 \\
			\hline
			$64\times 64$	&3.0019e-03       &2.00	    &5.1851e-02	     &1.00	  \\
			\hline
			$128\times 128$	&7.5102e-04       &2.00	    &2.5920e-02	     &1.00	  \\
			\hline
		\end{tabular}
	}	
\end{table}	

\begin{table}[H]
	\normalsize
	\caption{History of convergence with $ k=2 $: Example \ref{E6.2} }
	\label{T7}
	\centering
	\footnotesize
	\subtable[Method: HDG-I$ (l=2) $]{
		\begin{tabular}{p{1.5cm}<{\centering}|p{1.6cm}<{\centering}|p{0.8cm}<{\centering}|p{1.6cm}<{\centering}|p{0.8cm}<{\centering}}
			\hline   
			\multirow{2}{*}{mesh}&  
			\multicolumn{2}{c|}{$\frac{\lVert u(T)-u_{h}(T) \rVert_0}{\lVert  u(T)\rVert_0}$ }&\multicolumn{2}{c}{$\frac{\lVert \bm{q}(T)- \bm{q}_{h}(T)\rVert_0}{\lVert \bm{q}(T)\rVert_0}$}\cr\cline{2-5}    			
			&error&order&error&order\cr  
			\cline{1-5}
			$ 8\times 8 $	&2.8115e-02      &--	    &4.9545e-02	    &--    \\
			\hline
			$ 16\times 16 $	 &4.7515e-03     &2.56        &2.1126e-02       &1.23     \\
			\hline
			$32\times 32$	&6.1910e-04       &2.94	     &5.6085e-03      &1.91	 \\
			\hline
			$64\times 64$	&7.7864e-05       &2.99	    &1.4177e-03	     &1.98	  \\
			\hline
			$128\times 128$	&9.7479e-06       &3.00	    &3.5571e-04	     &1.99	  \\
			\hline
		\end{tabular}
	}
	
	\subtable[Method: HDG-II$ (l=1) $]{
		\begin{tabular}{p{1.5cm}<{\centering}|p{1.6cm}<{\centering}|p{0.8cm}<{\centering}|p{1.6cm}<{\centering}|p{0.8cm}<{\centering}}
			\hline   
			\multirow{2}{*}{mesh}&  
			\multicolumn{2}{c|}{$\frac{\lVert u(T)-u_{h}(T) \rVert_0}{\lVert  u(T)\rVert_0}$ }&\multicolumn{2}{c}{$\frac{\lVert \bm{q}(T)- \bm{q}_{h}(T)\rVert_0}{\lVert \bm{q}(T)\rVert_0}$}\cr\cline{2-5}  
			&error&order&error&order\cr  
			\cline{1-5}
			$ 8\times 8 $	&2.9126e-02      &--	    &5.0053e-02	    &--    \\
			\hline
			$ 16\times 16 $	 &4.8563e-03     &2.58       &2.1191e-02       &1.24     \\
			\hline
			$32\times 32$	&6.3425e-04       &2.94	     &5.6117e-03      &1.92	 \\
			\hline
			$64\times 64$	&7.9961e-05       &2.99	    &1.4199e-03	     &1.98	  \\
			\hline
			$128\times 128$	&1.0021e-05       &3.00	    &3.5663e-04	     &1.99	  \\
			\hline
		\end{tabular}
	}	
\end{table}	

\begin{exmp} \label{E6.3}
	This is a three-dimensional example to test the accuracy of the DIRK(2,3)   scheme  (\ref{5.16}) . We take $ \Omega=[0,1]\times[0,1]\times[0,1] $, $ T=1 $ and $ \nu=1 $. The exact solution to the problem (\ref{1.1a})-(\ref{1.1c}) is of the form 
	\begin{equation*}
	u=e^{-t}x(1-x)y(1-y)z(1-z),\ \ \ \ in\ \Omega\times[0,T].
	\end{equation*}
	%The source term $ f $ is then determined from this exact solution with . 
	We use $ N\times N\times N $ uniform triangular spatial meshes (c.f. Figure \ref{domain2}) for the computation. 
	
	To test the spatial accuracy, 
	 we take a   small time step $ \Delta t=0.005 $. Numerical results are given in Table 8 and 9 for $ k=1 $ and $ k=2 $, respectively, which show that   the scheme (\ref{5.16})
yields $(k+1)$-th and $(k)$-th spatial convergence orders of $\lVert u(T)-u_{h}(T)\rVert_0$ and $\lVert \bm{q}(T)-\bm{q}_{h}(T)\rVert_0$, respectively.

\end{exmp}	

\begin{table}[H]
	\normalsize
	\caption{History of convergence with $ k=1 $: Example \ref{E6.3} }
	\label{T8}
	\centering
	\footnotesize
	\subtable[Method: HDG-I$ (l=1) $]{
		\begin{tabular}{p{1.5cm}<{\centering}|p{1.6cm}<{\centering}|p{0.8cm}<{\centering}|p{1.6cm}<{\centering}|p{0.8cm}<{\centering}}
			\hline   
			\multirow{2}{*}{mesh}&  
			\multicolumn{2}{c|}{$\frac{\lVert u(T)-u_{h}(T) \rVert_0}{\lVert  u(T)\rVert_0}$ }&\multicolumn{2}{c}{$\frac{\lVert \bm{q}(T)- \bm{q}_{h}(T)\rVert_0}{\lVert \bm{q}(T)\rVert_0}$}\cr\cline{2-5}   			
			&error&order&error&order\cr  
			\cline{1-5}
			$ 2\times 2\times 2 $   &6.9815e-01      &--       &5.6066e-01     &--    \\
			\hline
			$ 4\times 4\times 4 $	&1.7672e-01      &1.98	    &3.0285e-01	    &0.89    \\
			\hline
			$ 8\times 8\times 8 $	 &4.4207e-02     &2.00        &1.5438e-01       &0.97     \\
			\hline
			$16\times 16\times 16$	&1.1050e-02       &2.00	     &7.7566e-02      &0.99	 \\
			\hline
			$32\times 32\times 32$	&2.7621e-03       &2.00	     &3.8830e-02      &1.00	 \\
			\hline
		\end{tabular}
	}
	
	\subtable[Method: HDG-II$ (l=0) $]{
		\begin{tabular}{p{1.5cm}<{\centering}|p{1.6cm}<{\centering}|p{0.8cm}<{\centering}|p{1.6cm}<{\centering}|p{0.8cm}<{\centering}}
			\hline   
			\multirow{2}{*}{mesh}&  
			\multicolumn{2}{c|}{$\frac{\lVert u(T)-u_{h}(T) \rVert_0}{\lVert  u(T)\rVert_0}$ }&\multicolumn{2}{c}{$\frac{\lVert \bm{q}(T)- \bm{q}_{h}(T)\rVert_0}{\lVert \bm{q}(T)\rVert_0}$}\cr\cline{2-5} 
			&error&order&error&order\cr  
			\cline{1-5}
			$ 2\times 2\times 2 $   &8.8064e-01      &--       &6.1105e-01     &--    \\
			\hline
			$ 4\times 4\times 4 $	&2.0917e-01      &2.07	    &3.1971e-01	    &0.93    \\
			\hline
			$ 8\times 8\times 8 $	 &5.1528e-02     &2.02        &1.6186e-01       &0.98     \\
			\hline
			$16\times 16\times 16$	&1.2835e-02       &2.00	     &8.1193e-02      &0.99	 \\
			\hline
			$32\times 32\times 32$	&3.2055e-03       &2.00	     &4.0630e-02      &1.00	 \\
			\hline
		\end{tabular}
	}	
\end{table}		

\begin{table}[H]
	\normalsize
	\caption{History of convergence with $ k=2 $: Example \ref{E6.3} }
	\label{T9}
	\centering
	\footnotesize
	\subtable[Method: HDG-I$ (l=2) $]{
		\begin{tabular}{p{1.5cm}<{\centering}|p{1.6cm}<{\centering}|p{0.8cm}<{\centering}|p{1.6cm}<{\centering}|p{0.8cm}<{\centering}}
			\hline   
			\multirow{2}{*}{mesh}&  
			\multicolumn{2}{c|}{$\frac{\lVert u(T)-u_{h}(T) \rVert_0}{\lVert  u(T)\rVert_0}$ }&\multicolumn{2}{c}{$\frac{\lVert \bm{q}(T)- \bm{q}_{h}(T)\rVert_0}{\lVert \bm{q}(T)\rVert_0}$}\cr\cline{2-5}   			
			&error&order&error&order\cr  
			\cline{1-5}
			$ 2\times 2\times 2 $   &1.3095e-01      &--       &1.9090e-01     &--    \\
			\hline
			$ 4\times 4\times 4 $	&1.4825e-02      &3.14	    &5.3052e-02	    &1.85    \\
			\hline
			$ 8\times 8\times 8 $	 &1.7531e-03     &3.08        &1.3659e-02       &1.96     \\
			\hline
			$16\times 16\times 16$	&2.1463e-04       &3.03	     &3.4459e-03      &1.99	 \\
			\hline
		\end{tabular}
	}
	
	\subtable[Method: HDG-II$ (l=1) $]{
		\begin{tabular}{p{1.5cm}<{\centering}|p{1.6cm}<{\centering}|p{0.8cm}<{\centering}|p{1.6cm}<{\centering}|p{0.8cm}<{\centering}}
			\hline   
			\multirow{2}{*}{mesh}&  
			\multicolumn{2}{c|}{$\frac{\lVert u(T)-u_{h}(T) \rVert_0}{\lVert  u(T)\rVert_0}$ }&\multicolumn{2}{c}{$\frac{\lVert \bm{q}(T)- \bm{q}_{h}(T)\rVert_0}{\lVert \bm{q}(T)\rVert_0}$}\cr\cline{2-5}  
			&error&order&error&order\cr  
			\cline{1-5}
			$ 2\times 2\times 2 $   &1.4705e-01      &--       &1.9582e-01     &--    \\
			\hline
			$ 4\times 4\times 4 $	&1.6034e-02      &3.20	    &5.3962e-02	    &1.86    \\
			\hline
			$ 8\times 8\times 8 $	 &1.8868e-03     &3.09        &1.3871e-02       &1.96     \\
			\hline
			$16\times 16\times 16$	&2.3104e-04       &3.03	     &3.4979e-03      &1.99	 \\
			\hline
			$32\times 32\times 32$	&2.8681e-05       &3.01	     &8.7713e-04      &2.00	 \\
			\hline
		\end{tabular}
	}	
\end{table}	

\begin{figure}[htp]	
	\centering
	\includegraphics[width=6.5cm,height=5cm]{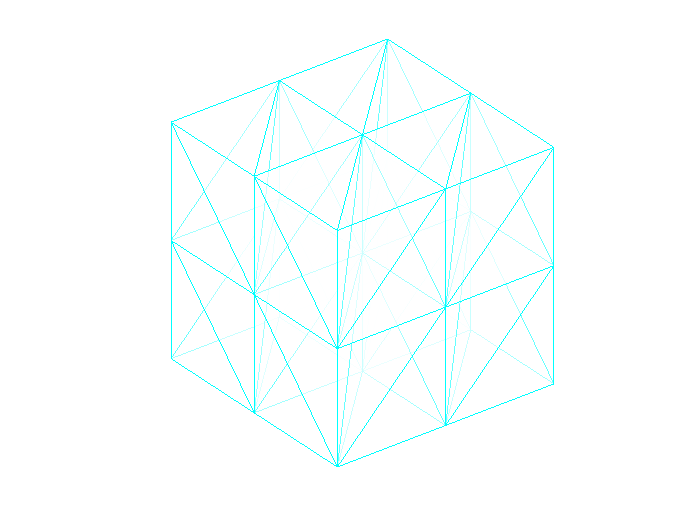}
	\includegraphics[width=6.5cm,height=5cm]{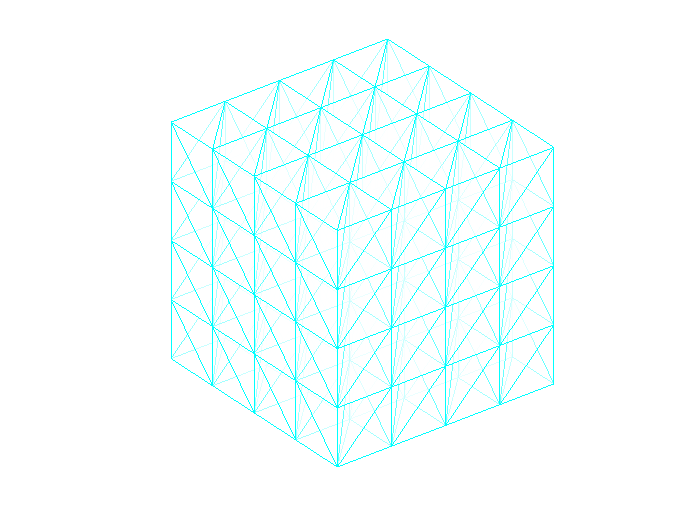} 
	\caption{The domain : $ 2\times 2\times 2 $(left) and $ 4\times 4\times 4 $(right) mesh}\label{domain2}
\end{figure}

\section{Conclusion}
In this paper, we have developed a class of semi-discrete and fully discrete  HDG  methods for the Burgers' equation in two and three dimensions. The existence
and uniqueness of the the semi-discrete solution and error estimation for the semi-discrete
and fully discrete schemes have been derived. Finally, numerical experiments have verified the theoretical results.

\bibliographystyle{plain}
\bibliography{ref}

\end{document}